\documentclass[12pt,twoside]{article}
\usepackage{bbm}
\usepackage{bm}
\usepackage{cite}
\usepackage{psfrag}
\usepackage{ifpdf}
\usepackage{multicol,graphics}
\usepackage{subfigure}
\usepackage[dvips]{graphicx}
\usepackage{color}
\usepackage{epsfig}
\usepackage{epstopdf}
\usepackage{float}
\usepackage{dsfont}
\usepackage{enumerate}
\usepackage[below]{placeins}
\usepackage[
  bookmarks=true,
  bookmarksopen=true,
  breaklinks=true,
  colorlinks=true,
  linkcolor=blue,
  anchorcolor=blue,
  citecolor=blue,
  filecolor=blue,
  menucolor=blue,
  urlcolor=blue
  ]{hyperref}
\usepackage{amsmath,amssymb,amsthm,mathrsfs,mathtools}

\newcommand{\norm}[1]{\displaystyle \left\| #1 \right\|}

\DeclareMathOperator{\dive}{div}
\DeclareMathOperator{\card}{card}
\allowdisplaybreaks[3]
\newtheorem{theorem}{Theorem}[section]
\newtheorem{lemma}[theorem]{Lemma}

\newtheorem{proposition}[theorem]{Proposition}

\newtheorem{definition}{Definition}[section]
\newtheorem{remark}{Remark}[section]

\newcounter{RomanNumber}

\numberwithin{equation}{section}
\begin{document}
\title{\textbf{\Large Existence of multi-dimensional pulsating fronts for KPP equations: a new formulation approach\\ }}
\author{Liangliang Deng\textsuperscript{a,b,}\footnote{The research of this author was supported by the China Scholarship Council.},
Arnaud Ducrot\textsuperscript{b,}\footnote{Corresponding author. E-mail address: arnaud.ducrot@univ-lehavre.fr}
 \\ \small \textsuperscript{a} School of Mathematics and Statistics, Lanzhou University, Lanzhou, Gansu 730000, China\\
 \small \textsuperscript{b} Normandie Univ, UNIHAVRE, LMAH, FR-CNRS-3335, ISCN, 76600 Le Havre, France}
\maketitle
\begin{abstract}
This paper is concerned with the existence of pulsating travelling fronts for a KPP reaction-diffusion equation posed in a multi-dimensional periodic medium.
We provide an alternative proof of the classic existence result. Our proof relies largely on the construction of a wave profile under a moving frame,
which avoids many technical difficulties in dealing with degenerate elliptic equations.
Intriguingly, our analysis also yields that the profile of the front propagating along each rational direction in $\mathbb{S}^{N-1}$ is periodic in time.


\end{abstract}

\section{Introduction}\label{introduction}
In this paper, we consider a heterogeneous reaction-diffusion equation of the form
\begin{align}\label{general-equation}
u_t-\dive(A(x)\nabla u)=f(x,u),~~t\in\mathbb{R},~x\in\mathbb{R}^N,
\end{align}
where $N$ is some given positive integer and without loss of generality, up to some change of variables, the spatial heterogeneities are assumed to be $\mathbb{Z}^{N}$-periodic.
Let us make the mathematical assumptions more precise.
We denote by $\mathbb{T}^{N}=\mathbb{R}^N/\mathbb{Z}^N$ the $N$-dimensional unit torus.
Suppose that the diffusion matrix field $A:\mathbb{T}^{N}\to\mathcal{S}_{N}(\mathbb{R})$ is of the class $C^{1+\alpha}$ for some exponent $\alpha\in(0,1)$ and uniformly elliptic in the sense that there exist two positive constants $0<\gamma\leq\Gamma$ such that
\begin{equation}\label{elliptic-condition}
\gamma\|\xi\|^2\leq\xi^{T}A(x)\xi
\leq\Gamma\|\xi\|^2,~~\forall (x,\xi)\in\mathbb{T}^{N}\times\mathbb{R}^{N}.
\end{equation}
Moreover, we assume that the nonlinearity $f:\mathbb{T}^{N}\times\mathbb{R}_+\to\mathbb{R}$ is continuous, of class $C^{\alpha}$ in $x$ locally uniformly in $u\in\mathbb{R}_+$ and of class $C^{1}$ in $u$ uniformly with respect to $x\in\mathbb{T}^{N}$ and we set $f_{u}(x,0):=\lim_{s\to0^{+}}f(x,s)/s$.
Further, the following hypotheses are satisfied:
\begin{equation}\label{reaction-hypothesis}
\begin{dcases}
f\geq0\mbox{ is of class }C^{1+\theta}\mbox{ (for some }\theta>0)\mbox{ with respect to }u\mbox{ in }\mathbb{\mathbb{T}}^{N}\times[0,1],\\
\forall x\in\mathbb{T}^N,~~f(x,0)=f(x,1)=0,\\
\forall x\in\mathbb{T}^N,~~s\mapsto\frac{f(x,s)}{s}\mbox{ is decreasing in } s>0.
\end{dcases}
\end{equation}
We remark that the last assumption in \eqref{reaction-hypothesis} implies that $f$ is of KPP type, that is,
$$
\forall (x,s)\in\mathbb{T}^N\times(0,1),~~
0<f(x,s)\leq f_u(x,0)s.
$$
Furthermore, one has $f_u(\cdot,1)<0$ and $f(x,s)<0$ for any $x\in\mathbb{T}^N$ and for all $s>1$. Notice that if \eqref{general-equation} admits a periodic stationary state $0\leq p(x)\leq1$ for all $x\in\mathbb{T}^N$, it then follows from the strong maximum principle and $f(\cdot,p)\geq0$ that $p(x)\equiv0$ or $p(x)\equiv1$.

Equations of the type \eqref{general-equation} arise in population genetics, the combustion theory and in spatial ecology (see \cite{SK-book} for instance). The archetype of such reaction-diffusion equations reads as the well-known homogeneous equation of the form
$$
u_t-\Delta u=u(1-u)~~~\mbox{ in }\mathbb{R}^N,
$$
which was introduced in the pioneering papers of Fisher \cite{F1937} and of Kolmogorov, Petrovsky and Piskunov \cite{KPP1937}.
The propagation phenomena of such a homogeneous equation have been widely studied in the literatures.
This includes in particular the so-called \emph{planar} travelling fronts, connecting $0$ and $1$, which are solutions of the form $u(t,x)=U(x\cdot e-ct)$.
This means that this particular solution propagates in a given direction $e\in\mathbb{S}^{N-1}$ (the unit sphere of $\mathbb R^N$) with a constant speed $c$, and that its profile is invariant in this moving frame.
We refer to Aronson and Weinberger \cite{AW1978} for more results about travelling fronts and spreading properties for some solutions to the Cauchy problem.

In the recent decades, much more attentions have been paid to the study of the propagation phenomena for reaction-diffusion problems posed in some heterogeneous media, typically of the type \eqref{general-equation}.
It can be traced back to the work of Freidlin and G\"{a}rtner \cite{FG1979}, who used a stochastic method to study the spreading properties for Fisher-KPP equations in the one-dimensional periodic environments.
Shigesada, Kawasaki and Teramoto \cite{SKT1986} first defined the notion of \emph{pulsating travelling fronts} (see also \cite{SK-book}),
which extends those planar fronts to spatially periodic environments.
They carried out some formal arguments and numerical simulations to study the critical fronts for one-dimensional Fisher-KPP equations,
where mobilities and nonlinearities vary with alternatively aligned patches. In general, we have
\begin{definition}[\textbf{Pulsating travelling front}]
\label{PTF-definition}\upshape
Let $e\in\mathbb{S}^{N-1}$ be an arbitrarily given vector.
An entire (classical) solution $u\equiv u(t,x)$ of \eqref{general-equation} is said be a \emph{pulsating travelling front} solution of \eqref{general-equation} connecting $0$ and $1$,
propagating in the direction $e$ with the effective speed $c\neq0$, if it satisfies
\begin{equation}\label{PTF}
 \begin{dcases}
  u\left(t+\frac{k\cdot e}{c},x\right)=u(t,x-k),~~\forall
   (t,x,k)\in\mathbb{R}\times\mathbb{R}^{N}\times\mathbb{Z}^{N},\\
   \lim\limits_{r\to+\infty}u(t,re+y)=0~\mbox{ and }
   \lim\limits_{r\to-\infty}u(t,re+y)=1,
 \end{dcases}
\end{equation}
where $y\in e^\perp:=\{\eta\in\mathbb{R}^{N}\mid \eta\cdot e=0\}$ and the above limits hold locally uniformly for $t\in\mathbb{R}$ and uniformly with respect to $y\in e^\bot$.
\end{definition}
It is easy to check that a \emph{pulsating travelling front} $(c,u)$ of \eqref{general-equation} propagating in the direction $e\in\mathbb{S}^{N-1}$ with speed $c\neq0$ can be equivalently written as $u(t,x)=U(s,x)$ with $s=x\cdot e-ct\in\mathbb{R}$, where the profile $U\colon\mathbb{R}\times\mathbb{T}^{N}\to\mathbb{R}$ satisfies the following semi-linear degenerate elliptic equation
\begin{align}\label{degenerate-elliptic-equation}
 \left(e\partial_s+\nabla_x\right)^{T}\left[A(x)\left(e\partial_s+
\nabla_x\right)U\right]+c\partial_s U
+f(x,U)=0,~~\forall(s,x)\in\mathbb{R}\times\mathbb{T}^{N},
\end{align}
as well as the asymptotic conditions
$$
 \lim_{s\to+\infty}U(s,x)=0~\mbox{ and }\lim_{s\to-\infty}U(s,x)=1
 \mbox{ uniformly for }x\in\mathbb{T}^{N}.
$$
Such a wave profile was first constructed by Xin \cite{X1991,X1992} in the framework of flame propagation. It is also called a \emph{periodically varying wavefront} in \cite{HZ1995}. Now it is widely adopted as the pulsating travelling wave in spatially periodic media (see \cite{BH2002,BHN2005,BHR2005-2,D2016,Hamel2008,LZ2010,W2002} for various type of nonlinearities and references therein).

Before going further let us outline some methods in the literatures to prove the existence of pulsating travelling fronts for Problem \eqref{general-equation} with monostable nonlinearities:

The first method is to solve the semi-linear degenerate elliptic equation \eqref{degenerate-elliptic-equation} using a regularization procedure. This approach was developed by Berestycki and Hamel \cite{BH2002} for the equation \eqref{general-equation} with an advection under a more general periodic framework (see also Berestycki, Hamel and Roques \cite{BHR2005-2}).
The authors first proved the existence of the solution $(c^{\varepsilon},U^{\varepsilon})$ for the regularized problem in  a cylinder with periodic boundary conditions:
\begin{equation*}
  \begin{cases}
   L_{\varepsilon}U+c\partial_s U+f(x,U)=0
   ~\mbox{ in }\mathbb{R}\times\mathbb{T}^{N}, \\
   \forall x\in\mathbb{T}^{N}, ~~U(-\infty,x)=1,~ U(\infty,x)=0
  \end{cases}
\end{equation*}
with
$$
L_{\varepsilon}U:=\varepsilon U_{ss}+\left(e\partial_s+\nabla_x\right)^{T}\left[A(x)\left(e\partial_s+
\nabla_x\right)U\right]\text{ for }\varepsilon>0.
$$
In the next step, they removed the elliptic regularization by passing to the limit $\varepsilon\to0$.
This approach requires refined estimates of the solution independent of $\varepsilon\ll 1$ small enough.

The second method we can mention is a general theory for the spreading speeds and the pulsating travelling waves developed by Weinberger \cite{W2002} using dynamical system arguments.
He proved the existence of spreading speed and its coincidence with the minimal speed of travelling waves for a recursion equation governed by an order-preserving compact operator of monostable type in a multi-dimensional periodic habitat. Liang and Zhao \cite{LZ2010} further generalized this theory to the abstract monotone semiflows with weak ($\alpha$-contraction) compactness.
These abstract results can be applied to various monotone systems posed on periodic environments to describe the spreading speed properties and to ensure the existence of pulsating fronts.

A third method we can also mention is based on the intersection number argument for one-dimensional reaction-diffusion equations, for which we refer to \cite{Angenent, Matano}.
Such a method has been proposed in \cite{DGM2014} and further developed in \cite{GM2020}.
Under weak assumptions, it is in particular proved that the solution starting from Heaviside initial data is steeper than any entire solution.
This allows the authors to obtain the existence of pulsating fronts for monostable equations while bistable and combustion cases are also considered.
The approach has also been extended by Nadin \cite{N2015} for some problems posed on a general heterogeneous medium, namely without any periodic assumptions.

Let us finally mention an other method recently developed by Griette and Matano in \cite{Griette-Matano-2021} for an one-dimensional spatially periodic reaction-diffusion system with hybrid nonlinearity.
For such systems, pulsating waves with speed $c>0$ are reformulated as a fixed point of the time $T=\frac{1}{c}$ Poincar\'e map composed with a spatial shift corresponding to the period of the environment. This fixed point equation is roughly solved using the Schauder fixed point theorem.
This methodology seems to be well adapted to handle the existence of pulsating waves in one-dimensional environments. The case of multi-dimensional media seems to be more complicated to deal with using such an approach.

In this work we aim to devise another approach to discuss the existence of pulsating fronts for multi-dimensional reaction-diffusion equations.
To illustrate our method, we consider a periodic Fisher-KPP reaction-diffusion equation, as in \eqref{general-equation} and we provide a new proof for the existence of the multi-dimensional pulsating travelling fronts.
Instead of working with a degenerate elliptic equation for the wave profile (as in \eqref{degenerate-elliptic-equation}), we propose an other formulation of this profile that satisfies a nondegenerate and parabolic equation, with periodicity (both in time and space) conditions for a dense set of directions.
Hence the difficulties caused by the elliptic degeneracy are successfully overcome for a dense subset of direction of propagation, while the case of a general direction of propagation is obtained by limiting arguments.
As far as we know, the existence problem of multi-dimensional pulsating travelling waves is scarcely studied for spatially periodic reaction-diffusion systems of the epidemic or prey-predator types.
In our forthcoming paper \cite{forthcoming}, we will use the methodology developed in this note to treat this issue for some epidemic systems with diffusion in a periodic medium.

Finally, let us mention that the notion of pulsating travelling front in spatially periodic media has been extended to more general (space and time) heterogeneous media.
One may refer to the notion of generalized pulsating travelling fronts introduced by Nolen et al. \cite{NRX2005} and Nadin \cite{N2009} or almost pulsating waves by Fang et al. \cite{FYZ2017} for monostable equations in space-time periodic media.
We also mention that another definition of travelling waves for nonlocal-dispersal monostable equations was introduce by Shen and Zhang \cite{SZ2012} in spatially periodic media and by Rawal et al. \cite{RSZ2015} in space-time periodic media.
We refer to Matano \cite{Matano03} for a definition of generalized travelling waves in a general random medium and to
Berestycki and Hamel \cite{BH2012} who introduced the notion of generalized transition waves for very general heterogeneous reaction-diffusion-advection equations.

The work is organized as follows:
Section \ref{new-moving-frame} gives a precise description of our methodology.
In Section \ref{preliminaries}, we recall some properties of a periodic elliptic eigenvalue problem.
Section \ref{proof} is concerned with the proof of existence results, which involves the building of an invariant domain, the derivation of a space-time periodic problem, rational approximation to any directions of propagation and the existence of the minimal wave speed.

\section{Description of a new wave profile}\label{new-moving-frame}
The aim of this section is to elaborate on our methodology. Assume that an entire solution $u=u(t,x)$ of the following equation
\begin{align}\label{periodic-KPP-equation}
u_t-\dive(A(x)\nabla u)=f(x,u),~~t\in \mathbb R,\;x\in\mathbb{R}^{N}
\end{align}
is a pulsating travelling front of \eqref{periodic-KPP-equation} propagating in a given direction $e\in\mathbb{S}^{N-1}$ with the effective speed $c\neq0$ according to Definition \ref{PTF-definition}.
We denote by $\{e_1,\cdots,e_N\}$ the canonical basis of $\mathbb R^N$ and we consider a linear orthogonal transformation $R\in\mathcal{O}(\mathbb{R}^{N})$ (the corresponding matrix representation still uses the same notation below) such that
\begin{equation*}
\begin{split}
&Re_1=e\text{ and }\\
&e^\perp={\rm span}\;\left\{Re_2,\cdots,Re_N\right\}=\left\{\eta\in\mathbb{R}^{N}\mid
\eta=R(0,y)^{T},\forall y\in\mathbb{R}^{N-1}\right\}.
\end{split}
\end{equation*}
For any $k\in\mathbb{Z}^{N}$, we decompose $k$ along $e$ and $e^\perp$, using the following notations
$$
k=(k\cdot e)e+k_\perp~\text{ with $k_\perp=k-(k\cdot e)e\in e^\perp$},
$$
so that
$$
R^{-1}k=
\begin{pmatrix}
k\cdot e\\
0_{\mathbb{R}^{N-1}}
\end{pmatrix}
+R^{-1}k_\perp
~\text{ and }~
R^{-1}k_\perp=
\begin{pmatrix}
 0\\
 \hat R k_\perp
\end{pmatrix}
\in \bigoplus_{i=2}^N \mathbb R e_i=\{0\}\times\mathbb R^{N-1}
$$
for some linear map $\hat R:e^\perp\to\mathbb{R}^{N-1}$. Now we consider the frame moving with the speed $c$ and we define the function $\varphi=\varphi(\xi,t,y)$ for $(\xi,t,y)\in
\mathbb{R}\times\mathbb{R}\times\mathbb{R}^{N-1}$ by
\begin{equation}\label{def-phi}
\varphi(\xi,t,y)=u(t,x)~\text{ with }x=R\begin{pmatrix}\xi+ct\\ y\end{pmatrix}.
\end{equation}
From the pulsating condition of $u$ in \eqref{PTF}, one has for all $k\in \mathbb Z^N$
\begin{align*}
 \varphi(\xi,t,y)&=u\left(t+\frac{k\cdot e}{c}, x+k\right)\\
&=u\left(t+\frac{k\cdot e}{c}, R\left((\xi+ct,y)^T+R^{-1}k\right)\right)\\
&=u\left(t+\frac{k\cdot e}{c}, R\left(\left(\xi+c\left(t+\frac{k\cdot e}{c}\right),y\right)^T+R^{-1}k_\perp\right)\right)\\
&=\varphi\left(\xi,t+\frac{k\cdot e}{c},y+\hat{R}k_\perp\right).
\end{align*}
Hence $\varphi$ satisfies what we shall call below the \emph{$R-$pulsating condition}
\begin{equation}\label{R-pulsating}
\varphi(\xi,t,y)=\varphi\left(\xi,t+\frac{k\cdot e}{c},y+\hat{R}k_\perp\right),\;\forall (\xi,t,y)\in
\mathbb{R}\times\mathbb{R}\times\mathbb{R}^{N-1},\;\forall k\in \mathbb Z^N.
\end{equation}
Next, recalling the definition of $\varphi$ in \eqref{def-phi} one has
$$
\frac{\partial u}{\partial x_i}(t,x)=\frac{\partial\varphi}{\partial \xi}
\frac{\partial \xi}{\partial x_i}+\sum^{N-1}_{j=1}\frac{\partial\varphi}{\partial y_j}
\frac{\partial y_j}{\partial x_i}=\nabla_{(\xi,y)}\varphi\cdot\frac{\partial}
{\partial x_i}
\begin{pmatrix}{\xi}\\{y}\end{pmatrix}
=\nabla_{(\xi,y)}\varphi(\xi,t,y)\cdot(R^{-1}e_i),
$$
and
$$
\frac{\partial^{2}u}{\partial x_i\partial x_j}(t,x)=D^2_{(\xi,y)}\varphi(\xi,t,y)\cdot (R^{-1}e_i)(R^{-1}e_j).
$$
To go further in this computation, we use a single variable $z=(z_1,\ldots,z_N)\in\mathbb{R}^N$ instead of $(\xi,y)\in\mathbb{R}\times \mathbb R^{N-1}$ and we set
\begin{align}\label{single-variable}
\tilde{\varphi}(t,z)=\varphi(\xi,t,y)~~\text{ with }
z_1=\xi,~(z_2,\ldots,z_{N})=y.
\end{align}
Using this notation we have
$$
\frac{\partial^{2}u}{\partial x_i\partial x_j}(t,x)=\sum_{\ell,k=1}^{N}
\frac{\partial^{2}\tilde \varphi}{\partial z_{\ell}\partial z_k}
\left(R^{-1}e_i\right)_k\left(R^{-1}e_j\right)_{\ell},
$$
wherein $\left(R^{-1}e_i\right)_k$ and $\left(R^{-1}e_j\right)_{\ell}$ denote the coordinate of the vectors $R^{-1}e_i$ and $R^{-1}e_j$ in the canonical basis of $\mathbb R^N$, respectively.
Further, we set
$$
\widetilde{A}(\xi,t,y)=\left(\tilde{a}_{ij}(\xi,t,y)\right)
=R A\left(R(\xi+ct,y)^{T}\right)R^T.
$$
Since $R\in\mathcal{O}(\mathbb{R}^{N})$ and the diffusion matrix field $A(x)=\left(a_{ij}(x)\right)_{1\leq i,j\leq N}$ is symmetric and uniformly elliptic in $\mathbb{T}^N$, so is $\widetilde{A}$. Moreover, let us write
$$
R=\begin{pmatrix}L_1\\\vdots \\L_{N}\end{pmatrix}\mbox{ and }~ R^{-1}=R^T=\left(L_1^T,\ldots,L_{N}^T\right),
$$
where $L_1$, ..., $L_N$ are the lines of the matrix $R$, so that $L_1^T$, ..., $L_N^T$ are the columns of $R^{-1}=R^T$.
Next, we have
\begin{align*}
 I:=\sum_{i,j=1}^{N}a_{ij}(x)
 \frac{\partial^{2}u}{\partial x_i\partial x_j}
 &=\sum_{i,j=1}^{N}a_{ij}\left(R(\xi+ct,y)^{T}\right)\sum^{N}_{\ell,k=1}
 \frac{\partial^{2}\tilde \varphi}
{\partial z_{\ell}\partial z_k}\left(R^{-1}e_i\right)_{k}
\left(R^{-1}e_j\right)_{\ell}.
\end{align*}
Here $\left(R^{-1}e_i\right)_{k}$ denotes the $k^{\rm th}$ component of the vector $R^{-1}e_i=L_i^T$.
Hence one has
\begin{align*}
I&=\sum^{N}_{\ell,k=1}\frac{\partial^{2}\tilde{\varphi}}
{\partial z_{\ell}\partial z_k}
\sum^{N}_{i,j=1}a_{ij}\left(R(\xi+ct,y)^{T}\right)
\left(R^{-1}e_i\right)_{k}\left(R^{-1}e_j\right)_{\ell}\\
&=\sum^{N}_{\ell,k=1}\left(L_{\ell}A\left(R(\xi+ct,y)^{T}\right)L_k^T\right)
\frac{\partial^{2}\tilde{\varphi}}{\partial z_{\ell}\partial z_k}\\
&=\sum^{N}_{\ell,k=1}\left(e_{\ell}^T R A\left(R(\xi+ct,y)^{T}\right)
R^Te_k\right)\frac{\partial^{2}\tilde{\varphi}}
{\partial z_{\ell}\partial z_k}\\
&=\sum^{N}_{\ell,k=1}\tilde{a}_{{\ell}k}(\xi,t,y)
\frac{\partial^{2}\tilde{\varphi}}{\partial z_{\ell}\partial z_k},
\end{align*}
while
\begin{align*}
 J:=\sum_{i,j=1}^{N}\frac{\partial a_{ij}}{\partial x_j}
 \frac{\partial u}{\partial x_i}
 &=\sum_{i,j=1}^{N}\sum_{k=1}^{N}L_i\frac{\partial\widetilde{A}}
{\partial z_k}L_j\left(R^{-1}e_j\right)_{k}
\sum^{N}_{\ell=1}\frac{\partial\tilde{\varphi}}{\partial z_{\ell}}
\left(R^{-1}e_i\right)_{\ell}\\
&=\sum_{i,j=1}^{N}\sum_{k=1}^{N}\sum_{m,n=1}^{N}\frac{\partial\tilde{a}_{mn}}
{\partial z_k}(e_iR)_m(e_jR)_n\left(R^{-1}e_j\right)_{k}
\sum^{N}_{\ell=1}\frac{\partial\tilde{\varphi}}{\partial z_{\ell}}
\left(R^{-1}e_i\right)_{\ell}\\
&=\sum_{\ell,k=1}^{N}\frac{\partial\tilde{\varphi}}{\partial z_{\ell}}
\sum_{m,n=1}^{N}\frac{\partial\tilde{a}_{mn}}{\partial z_k}
\sum_{i,j=1}^{N}(e_iR)_m(e_jR)_n\left(R^{-1}e_j\right)_{k}
\left(R^{-1}e_i\right)_{\ell}\\
&=\sum_{\ell,k=1}^{N}\frac{\partial\tilde{\varphi}}{\partial z_{\ell}}
\sum_{m,n=1}^{N}\frac{\partial\tilde{a}_{mn}}{\partial z_k}
\langle L_m,L_{\ell}\rangle\langle L_n,L_k\rangle.
\end{align*}
Since the matrix $R$ is orthogonal, its lines are orthonormal so that for all $m,\,\ell\in\{1,\ldots,N\}$, one has
$$
\langle L_m,L_{\ell}\rangle=\delta_{m{\ell}}=
\begin{cases}
1~~\text{ if $m=\ell$,}\\
0~~\text{ else.}
\end{cases}
$$
Thus we get
\begin{align*}
J&=\sum_{\ell,k=1}^{N}\frac{\partial\tilde{\varphi}}{\partial z_{\ell}}
\sum_{m,n=1}^{N}\frac{\partial\tilde{a}_{mn}}{\partial z_k}
\delta_{m{\ell}}\delta_{nk}=\sum_{\ell,k=1}^{N}
\frac{\partial\tilde{\varphi}}{\partial z_{\ell}}
\frac{\partial\tilde{a}_{{\ell}k}}{\partial z_k}.
\end{align*}
Recalling $z=(\xi,y)\in\mathbb{R}\times\mathbb{R}^{N-1}$ in \eqref{single-variable}, the two formulae above imply that
$$
\dive\left(A(x)\nabla u\right)(t,x)=I+J=
\dive_{\xi,y}\big(\widetilde{A}(\xi,t,y)\nabla_{\xi,y}\varphi\big)(\xi,t,y).
$$
Consequently, the function $\varphi=\varphi(\xi,t,y)$ defined in \eqref{def-phi} becomes a solution of the equation
\begin{align}\label{profile-equation}
\varphi_t-\dive_{\xi,y}\big(\widetilde{A}(\xi,t,y)\nabla_{\xi,y}\varphi\big)
-c\varphi_{\xi}=\widetilde{f}(\xi,t,y,\varphi)
\end{align}
posed for $(\xi,t,y)\in\mathbb{R}\times\mathbb{R}\times\mathbb{R}^{N-1}$, wherein the diffusion matrix field $\widetilde{A}$ is given by
$$
\widetilde{A}(\xi,t,y)=\left(\tilde{a}_{ij}(\xi,t,y)\right)
=RA\left(R(\xi+ct,y)^{T}\right)R^T,
$$
and we have defined the nonlinear function $\widetilde f$ by
$$
\widetilde{f}(\xi,t,y,\varphi)=f\left(R(\xi+ct,y)^{T},\varphi\right).
$$
Therefore, when $u\equiv u(t,x)$ is a pulsating travelling front of \eqref{periodic-KPP-equation} in the direction $e\in \mathbb S^{N-1}$ with the speed $c>0$, the function $\varphi$ given by
\begin{align*}
  \varphi(\xi,t,y)=u\left(t,R(\xi+ct,y)^{T}\right),~~(\xi,t,y)\in
\mathbb{R}\times\mathbb{R}\times\mathbb{R}^{N-1}
\end{align*}
becomes an entire solution of \eqref{profile-equation} and the pulsating condition for $u$ in \eqref{PTF} rewrites as the following property
\begin{align}\label{profile-periodicity}
 \varphi\left(\xi,t+\frac{k\cdot e}{c},y+\hat{R}k_\perp\right)= \varphi(\xi,t,y),~~\forall (\xi,t,y)\in\mathbb{R}\times\mathbb{R}
\times\mathbb{R}^{N-1},~\forall k\in\mathbb Z^N.
\end{align}
The above condition can be referred as the $R-$pulsating condition as in \eqref{R-pulsating}.

Note also that since $A$ and $f$ are both $\mathbb Z^N$-periodic in $x$, the same $R-$pulsating condition as \eqref{profile-periodicity} is shared by $\widetilde{A}$ and $\widetilde{f}$, namely for all $k\in\mathbb{Z}^{N}$ and for all $(\xi,t,y)\in\mathbb{R}\times\mathbb{R}\times\mathbb{R}^{N-1}$,
\begin{align*}
 &\tilde{a}_{ij}\left(\xi,t+\frac{k\cdot e}{c},y+\hat{R}k_\perp\right)
=\tilde{a}_{ij}(\xi,t,y),~~\forall\,1\leq i,j\leq N,\\
&\widetilde{f}\left(\xi,t+\frac{k\cdot e}{c},y+\hat{R}k_\perp,s\right)
=\widetilde{f}(\xi,t,y,s),~~\forall s\geq0.
\end{align*}
As a consequence of the above analyses, we conclude that $(c,u)$ is a pulsating travelling front connecting $0$ and $1$ of \eqref{periodic-KPP-equation} propagating in the direction $e\in\mathbb{S}^{N-1}$ if and only if $(c,\varphi)$ is a solution of \eqref{profile-equation} which satisfies \eqref{profile-periodicity} as well as the asymptotic conditions
\begin{align}\label{profile-limits}
 \lim_{\xi\to+\infty}\varphi(\xi,t,y)=0~\mbox{ and } \lim_{\xi\to-\infty}\varphi(\xi,t,y)=1
 \mbox{ uniformly for }(t,y)\in\mathbb{R}\times\mathbb{R}^{N-1}.
\end{align}
Roughly speaking, when $(c,u)$ is a pulsating travelling front of Problem \eqref{periodic-KPP-equation}, then $(c,\varphi)$ becomes a usual travelling wave solution of the uniformly parabolic problem \eqref{profile-equation} propagating inside the cylinder $\mathbb{R}\times\mathbb{R}^{N-1}$ with an infinite section.
\begin{remark}\label{profile-equivalence}
  The profile $(\xi,t,y)\in \mathbb{R}\times\mathbb{R}\times\mathbb{R}^{N-1}\mapsto\varphi(\xi,t,y)$ introduced above is also equivalent to the profile stated in Section \ref{introduction}. Indeed, one can easily check that
  $$
  \varphi(\xi, (ct,y))=\Phi\left(\xi, R(ct,y)^T\right),
  $$
  where the function $\Phi: \mathbb{R}\times\mathbb{R}^N\to\mathbb{R}$ is $\mathbb Z^N$-periodic in the second variable, namely
  \begin{align*}
    \Phi(\xi,X+k)=\Phi(\xi,X),~~\forall (\xi,X)\in\mathbb R\times \mathbb R^N,~ \forall k\in \mathbb Z^N.
  \end{align*}
\end{remark}
\begin{remark}\label{profile-not-periodic}
Note that under the frame moving with the speed $c$, an one-dimensional pulsating wave $(c,u)$ of \eqref{periodic-KPP-equation} is its an entire solution such that for all $\xi\in\mathbb{R}$, the function $t\mapsto u(t,\xi+ct)$ is $1/c$-periodic.
A special case we should mention is the work of Bages et al. in \cite{BMJ-TAMS} where the orthogonal transformation $R$ can be regarded as an identity transformation in $\mathbb{R}^N$, namely the pulsating front propagates along the direction $e_{1}=(1,0_{\mathbb{R}^{N-1}})$.
In this case, the profile of the front is still $1/c$-periodic in time, whence the existence results can be readily obtained with the help of the Poincar\'e map. However, it can be seen from \eqref{profile-periodicity} that such observations generally no longer hold for the multi-dimensional pulsating fronts propagating in an arbitrary direction in $\mathbb{S}^{N-1}$.
\end{remark}
Let us mention that our methodology can be extended to more general equations \eqref{periodic-KPP-equation} with a smooth advection $V:\mathbb{T}^N\to\mathbb{R}^N$. In this case, we set $\widetilde{V}(\xi,t,y)=V(R(\xi+ct,y)^{T})R$. Then, $\varphi$ satisfies the following parabolic equation
\begin{align}\label{advection}
\varphi_t-\dive_{\xi,y}\big(\widetilde{A}(\xi,t,y)\nabla_{\xi,y}\varphi\big)
+\big(\widetilde{V}(\xi,t,y)\cdot\nabla_{\xi,y}-c\partial_\xi\big)\varphi
=\widetilde{f}(\xi,t,y,\varphi).
\end{align}
In particular, this allows us to prove the existence of the pulsating front for equation \eqref{general-equation} with more general nonlinearities connecting $0$ and a unique positive stationary state $p\in C^2(\mathbb{T}^N)$ of \eqref{general-equation} (see \cite{BHR2005-2,W2002}). A typical example of such nonlinearities is
$$
g(x,u)=u(a(x)-u)
$$
wherein the function $a\in C^\alpha(\mathbb{T}^N)$ is not positive everywhere. However, without loss of generality, we can always work under the hypothesis \eqref{reaction-hypothesis}. In fact, if we set
$$
\widetilde{p}(\xi,t,y):=p(R(\xi+ct,y)^T)\mbox{ and }
\widetilde{g}(\xi,t,y,\varphi):=g(R(\xi+ct,y)^T,\varphi),~
R\begin{pmatrix}{\xi+ct}\\{y}\end{pmatrix}\in\mathbb{T}^N,
$$
and we write
$$
\psi(\xi,t,y)=\frac{\varphi(\xi,t,y)}{\widetilde{p}(\xi,t,y)},~~
\widetilde{f}(\xi,t,y,\psi)=\frac{1}{\widetilde{p}(\xi,t,y)}
\left[\widetilde{g}(\xi,t,y,\psi\widetilde{p})-\psi\widetilde{g}
(\xi,t,y,\widetilde{p})\right]
$$
and
$$
  \widetilde{V}(\xi,t,y)=-\frac{2}{\widetilde{p}(\xi,t,y)}
\left[\left(R^T\widetilde{A}(\xi,t,y)R\right)
\Big(R\nabla_{\xi,y}\widetilde{p}(\xi,t,y)\Big)\right]^TR,
$$
then the function $\psi$ satisfies the equation \eqref{advection} and the nonlinear function $f$ defined by $f(R(\xi+ct,y)^T,\psi):=\widetilde{f}(\xi,t,y,\psi)$ satisfies our hypothesis \eqref{reaction-hypothesis}.

In what follows we will adopt the methodology different from those in \cite{BH2002,BHR2005-2} and \cite{W2002} to prove the same results for KPP-type nonlinearity:
\begin{theorem}\label{main-result}
  Under the above assumptions on $A$ and $f$, for each $e\in\mathbb{S}^{N-1}$, there exists $c^*(e)>0$ such that for each $c\geq c^{*}(e)$, Problem \eqref{general-equation} has a pulsating travelling front solution $(c,u)$ propagating in the direction $e$ according to Definition \ref{PTF-definition}, while there exist no pulsating travelling fronts of speed $c$ for $c<c^*(e)$.
\end{theorem}

The focus of the present paper is to provide a more concise proof of the existence result based on the new formulation of the wave profile introduced above.
The nonexistence with small speeds is a direct consequence of spreading properties for some solutions to the Cauchy problem associated with \eqref{general-equation} (see \cite{W2002} for instance).
One may also refer to \cite{BH2002,BHR2005-2} for other proofs of non-existence results.

As already underlined, the main difference with the method in \cite{BH2002,BHR2005-2} is that we directly work with an equivalent uniformly parabolic problem \eqref{profile-equation}-\eqref{profile-limits} instead of a degenerate elliptic equation derived from the change of variables $(t,x)\mapsto(x\cdot e-ct, x)$ ($c\neq0$). In particular, our proof indicates that although it is known that for each given $e\in\mathbb{S}^{N-1}$ the profile of the pulsating fronts is in general quasi-periodic in time, in the sense that the function $t\mapsto u(t,x+cte)$ is quasi-periodic for all $x\in\mathbb{R}^N$, there are still enough pulsating travelling fronts such that its profile is periodic in time.
\section{Preliminaries}\label{preliminaries}
In this section, we recall some important properties of the periodic principal eigenvalue for an elliptic eigenvalue problem. Consider the linearized equation of \eqref{profile-equation} around $0$
$$
\varphi_t-\dive_{\xi,y}\big(\widetilde{A}(\xi,t,y)\nabla_{\xi,y}\varphi\big)
-c\varphi_{\xi}=\widetilde{f}_\varphi(\xi,t,y,0)\varphi
$$
where $\widetilde{f}_\varphi(\xi,t,y,0)=f_u(R(\xi+ct,y)^{T},0)$. For clarity, we denote by
$$X:=R(\xi+ct,y)^{T}$$
 the generic element of $\mathbb{T}^{N}$. We look for a solution to the above equation of the form
$$
\varphi(\xi,t,y)=e^{-\lambda \xi}\Phi(X),~~X\in\mathbb{T}^{N}.
$$
Then the pair $(\lambda,\Phi)\in\mathbb{R}\times C^{2}(\mathbb{T}^{N})$ satisfies the following elliptic equation
\begin{align}\label{speed-problem}
  L_{\lambda}\psi=c\lambda\psi~~~\mbox{ in }\mathbb{T}^{N},
\end{align}
where $L_{\lambda}$ is the elliptic operator of the form
$$
L_{\lambda}\psi:=\dive\left(A(X)\nabla\psi\right)-2\lambda eA(X)\nabla\psi+\left[\lambda^{2}eA(X)e-\lambda\dive(A(X)e)
+f_u(X,0)\right]\psi
$$
with periodicity conditions. To solve \eqref{speed-problem}, let us consider for each $\lambda\in\mathbb{R}$ and each $e\in\mathbb{S}^{N-1}$ the principal eigenvalue problem:
\begin{equation}\label{periodic-eigenvalue-problem}
 \begin{cases}
   L_{\lambda}\Phi_{\lambda,e}=k_\lambda(e)\Phi_{\lambda,e}~\mbox{ in } \mathbb{T}^{N}, \\
   \Phi_{\lambda,e}\in C^{2}(\mathbb{T}^{N}),~~\Phi_{\lambda,e}>0.
 \end{cases}
\end{equation}
It is known that $k_\lambda(e)$ is an algebraically simple eigenvalue and the principal eigenfunction $\Phi_{\lambda,e}$ is unique up to multiplication (see \cite{Pinsky1995} for instance). Let $\Phi_{\lambda,e}\in C^{2}(\mathbb{T}^{N})$ be the unique principal eigenfunction of \eqref{periodic-eigenvalue-problem} such that
\begin{align}\label{normalization}
 \Phi_{\lambda,e}>0~\mbox{ and }~\|\Phi_{\lambda,e}\|_{\infty}=1.
\end{align}
A few properties of these eigenelements $\left(k_{\lambda}(e),\Phi_{\lambda,e}\right)$ are collected in the following:
\begin{proposition}\label{principal-eigenelements}
The following properties hold:
\begin{enumerate}[(i)]
\item For each $e\in\mathbb{S}^{N-1}$, $k_\lambda(e)$ is analytic, convex with respect to $\lambda\in\mathbb{R}$ and bounded below away from zero.
\item For each $e\in\mathbb{S}^{N-1}$, the function $\lambda\mapsto k_\lambda(e)/\lambda$ is continuous on $(0,\infty)$ and satisfies
  $$
  \lim\limits_{\lambda\to0^{+}}\frac{k_\lambda(\cdot)}{\lambda}=+\infty
  ~\mbox{ and }~
  \liminf\limits_{\lambda\to+\infty}\frac{k_\lambda(\cdot)}{\lambda}>0.
  $$
  Furthermore, it can reach the minimum on $(0,\infty)$.
\item For any $\lambda\in\mathbb{R}$, the function $e\in\mathbb{S}^{N-1}\mapsto k_\lambda(e)$ is continuous.  Moreover, the principal eigenfunction $\Phi_{\lambda,e}(X)$ depends continuously on both $\lambda\in\mathbb{R}$ and $e\in\mathbb{S}^{N-1}$ with the uniform topology.
  \end{enumerate}
\end{proposition}
Thanks to this proposition, one can define the following two quantities
$$
c^{*}(e):=\min_{\lambda>0}\frac{k_\lambda(e)}{\lambda}>0~\mbox{ and }~
\lambda_c(e):=\min\{\lambda>0~|~ k_\lambda(e)-c\lambda=0\}.
$$
Furthermore, we have the following proposition.
\begin{proposition}\label{minimal-speed}
The following statements hold true:
\begin{enumerate}[(i)]
\item The equation
$$
k_\lambda(e)-c\lambda=0
$$
has solutions if and only if $c\geq c^{*}(e)$. Moreover, when $c>c^{*}(e)$, there are two positive solutions $\lambda^-_{c}<\lambda^+_{c}$ and when $c=c^{*}(e)$, there is the unique solution $\lambda^*>0$.
\item $c^{*}(e)$ is continuous with respect to $e\in\mathbb{S}^{N-1}$. Moreover, for any $c\geq c^{*}(e)$, $\lambda_c(e)$ is continuous and bounded with respect to $e\in\mathbb{S}^{N-1}$.
  \end{enumerate}
\end{proposition}
These results can be found in \cite{BH2002,BHN2005} and the last assertion in Proposition \ref{minimal-speed} can be readily deduced from \cite{BH2002,N2010} under KPP assumptions. We mention that the authors in \cite{AG2016} have shown the continuity of the function $e\in\mathbb{S}^{N-1}\mapsto c^*(e)$ in monostable (not necessarily with KPP assumption) and ignition cases.
\section{Existence of pulsating travelling fronts}\label{proof}
This section is devoted to the proof of Theorem \ref{main-result}.
Firstly, we investigate the existence of pulsating travelling fronts propagating along each direction with rational coordinates in $\mathbb{S}^{N-1}$ by solving a space-time periodic problem.
Next, a rational approximation will allow us to obtain the existence of pulsating travelling fronts propagating in any direction of propagation. Finally, we prove the existence of the minimal wave speed $c^*(e)$.
\subsection{Construction of sub- and supersolutions}\label{invariant-domain}
The goal of this subsection is to construct sub- and supersolutions of \eqref{profile-equation} with supercritical speeds and for all directions of propagation. For notational simplicity, we temporarily forget the dependence of $k_\lambda(e)$, $\Phi_{\lambda,e}$ and $\lambda_c(e)$ on $e\in \mathbb S^{N-1}$ in Subsections \ref{invariant-domain}-\ref{existence-rational-direction} below.
\begin{lemma}\label{subsolution}
For all $c>c^{*}(e)$, let $\psi^{-}$ be the function defined by
$$
\psi^{-}(\xi,t,y)=\Phi_{\lambda_{c}}(X)
e^{-\lambda_{c}\xi}-K\Phi_{\lambda_{c}+\delta}(X)
e^{-(\lambda_{c}+\delta)\xi},
$$
where the functions $\Phi_{\lambda_{c}}$ and $\Phi_{\lambda_{c}+\delta}$ are the unique principal eigenfunctions of problem \eqref{periodic-eigenvalue-problem} in the sense of \eqref{normalization} corresponding to $\lambda_{c}$ and $\lambda_{c}+\delta$, respectively. Then there exist some $\delta\in\left(0,\min\{\theta\lambda_c,\lambda^+_{c}-\lambda_{c}\}\right)$ small enough and $K>0$ large enough such that the function $\psi^{-}\leq1$ is a subsolution of the equation \eqref{profile-equation} on the set
$$
\Omega^+:=\{(\xi,t,y)\in\mathbb{R}\times\mathbb{R}
\times\mathbb{R}^{N-1}: \psi^-(\xi,t,y)>0\}.
$$
\end{lemma}
\begin{proof} By Propositions \ref{principal-eigenelements}-\ref{minimal-speed}, one has $k^{\prime}_{\lambda{_c}}<c$ for any $c>c^{*}(e)$ (the prime denotes the derivative with respect to $\lambda$) since $\lambda_c$ is the smallest positive solution such that $k_\lambda-c\lambda=0$, $k_0>0$ (due to $f_u(\cdot,0)>0$) and $k_\lambda$ is convex with respect to $\lambda$. Therefore, we can choose some $\delta\in(0,\lambda^+_{c}-\lambda_{c})$ small enough so that \begin{equation}\label{r_delta}
r_\delta:=c(\lambda_{c}+\delta)-k_{\lambda_{c}+\delta}>0.
\end{equation}
Moreover, we know from our hypothesis \eqref{reaction-hypothesis} that there exists some $\gamma>0$ such that
\begin{align}\label{subsolution-proof-1}
\widetilde{f}(\xi,t,y,s)\geq \widetilde{f}_{\varphi}(\xi,t,y,0)s-\gamma s^{1+\theta},~~\forall s\in[0,1],
~\forall(\xi,t,y)\in\mathbb{R}\times\mathbb{R}\times\mathbb{R}^{N-1}.
\end{align}
Let $K_1>0$ be large enough such that
\begin{align}\label{subsolution-proof-2}
\max_{(\xi,X)\in\mathbb{R}\times\mathbb{T}^{N}}
\left(\Phi_{\lambda_{c}}(X)
e^{-\lambda_{c}\xi}-K_1\Phi_{\lambda_{c}+\delta}(X)
e^{-(\lambda_{c}+\delta)\xi}\right)\leq1
\end{align}
and set
\begin{align*}
  K=\max\left\{K_1,\frac{\gamma}{r_\delta}\sup_{\mathbb{T}^N}
  \frac{\Phi^{1+\theta}_{\lambda_c}}{\Phi_{\lambda_{c}+\delta}},
  \sup_{\mathbb{T}^N}\frac{\Phi_{\lambda_c}}{\Phi_{\lambda_{c}+\delta}}\right\}.
\end{align*}
Then, we have
\begin{align}\label{subsolution-proof-3}
  \gamma\Phi^{1+\theta}_{\lambda_c}(X)\leq r_{\delta}K\Phi_{\lambda_{c}+\delta}(X),~~\forall X\in\mathbb{T}^N
\end{align}
and
$
\psi^-(\xi,t,y)\leq0$ for all $(\xi,t,y)\in(-\infty,0]\times\mathbb{R}^{N}$.
Now fix  $\delta\in\left(0,\min\{\theta\lambda_c,\lambda^+_{c}-\lambda_{c}\}\right)$. Due to \eqref{subsolution-proof-1}-\eqref{subsolution-proof-3}, we obtain that
\begin{align*}
 &\psi^-_t-\dive_{\xi,y}\big(\widetilde{A}(\xi,t,y)\nabla_{\xi,y}\psi^-\big)
 -c\psi^-_\xi\\
 =&f_u(X,0)\psi^{-}-K\left[c(\lambda_{c}+\delta)-
 k_{\lambda_{c}+\delta}\right]\Phi_{\lambda_{c}+\delta}(X)
e^{-(\lambda_{c}+\delta)\xi}\\
=&\widetilde{f}_{\varphi}(\xi,t,y,0)\psi^{-}-r_\delta K\Phi_{\lambda_{c}+\delta}(X)
e^{-(\lambda_{c}+\delta)\xi}\\
\leq& \widetilde{f}(\xi,t,y,\psi^{-})+\gamma(\psi^{-})^{1+\theta}-r_\delta K\Phi_{\lambda_{c}+\delta}(X)e^{-(\lambda_{c}+\delta)\xi}\\
\leq& \widetilde{f}(\xi,t,y,\psi^{-})+\gamma\Phi^{1+\theta}_{\lambda_c}(X)
e^{-(1+\theta)\lambda_{c}\xi}-r_\delta K\Phi_{\lambda_{c}+\delta}(X)
e^{-(\lambda_{c}+\delta)\xi}\\
\leq& \widetilde{f}(\xi,t,y,\psi^{-})+\left(\gamma\Phi^{1+\theta}_{\lambda_c}(X)-
r_\delta K\Phi_{\lambda_{c}+\delta}(X)\right)
e^{-(\lambda_{c}+\delta)\xi}\\
\leq& \widetilde{f}(\xi,t,y,\psi^{-})
\end{align*}
for all $(\xi,t,y)\in\Omega^+$, where the third inequality holds due to $\Omega^+\subset\mathbb{R}_+\times\mathbb{R}^{N}$ under our choice for the parameters $K$ and $\delta$. Thus, $\psi^{-}$ is a subsolution of \eqref{profile-equation} on the set $\Omega^+$.
\end{proof}
\begin{lemma}\label{supersolution}
For all $c\geq c^{*}(e)$, the function
$$
\psi^+(\xi,t,y)=\Phi_{\lambda_{c}}(X)
e^{-\lambda_{c}\xi}
$$
is a global supersolution of the equation \eqref{profile-equation}, where the function $\Phi_{\lambda_{c}}$ is the unique principal eigenfunction of \eqref{periodic-eigenvalue-problem} in the sense of \eqref{normalization} corresponding to $\lambda_{c}$.
\end{lemma}
\begin{proof}
To prove the lemma we compute
\begin{align*}
 &\left(\Phi_{\lambda_{c}}(X)e^{-\lambda_{c}\xi}\right)_t-
 \dive_{\xi,y}\left(\widetilde{A}(\xi,t,y)\nabla_{\xi,y}(\Phi_{\lambda_{c}}(X)
 e^{-\lambda_{c}\xi})\right)-
 c\left(\Phi_{\lambda_{c}}(X)e^{-\lambda_{c}\xi}\right)_\xi\\
 =&\left[c\lambda_c\Phi_{\lambda_{c}}-\dive\left(A(X)\nabla
 \Phi_{\lambda_{c}}\right)+2\lambda_{c}eA(X)\Phi_{\lambda_{c}}
 -\left(\lambda^2_{c}eA(X)e-\lambda_c\dive\left(A(X)e\right)\right)
 \Phi_{\lambda_{c}}\right]e^{-\lambda_{c}\xi}\\
 =&\widetilde{f}_{\varphi}(\xi,t,y,0)\Phi_{\lambda_{c}}e^{-\lambda_{c}\xi}
 \geq \widetilde{f}\left(\xi,t,y,\Phi_{\lambda_{c}}e^{-\lambda_{c}\xi}\right)
\end{align*}
for all $(\xi,t,y)\in\mathbb{R}\times\mathbb{R}\times\mathbb{R}^{N-1}$. This gives the conclusion.
\end{proof}
\subsection{Pulsating travelling fronts in each rational direction}
\label{existence-rational-direction}
In this subsection, we will prove the existence of pulsating travelling fronts for Problem \eqref{general-equation} propagating in each unit direction with rational coordinates with supercritical speeds under the framework of Section \ref{new-moving-frame}.
\subsubsection{The derivation of a space-time periodic problem}
The goal of this part is to transform Problem \eqref{profile-equation}-\eqref{profile-limits} into a space-time periodic problem. To this end, we first prove a result of independent interest, which plays a crucial role in our argument.
\begin{lemma}\label{change-basis}
  Assume that $\zeta\in\mathbb{Q}^N\cap\mathbb{S}^{N-1}$ for some $N\geq2$. Then there exist an orthogonal basis $\{\zeta_i\}_{i=1}^{N}$ of $\mathbb{R}^N$ and $N$ constants $\tau_i>0$ such that $\zeta_1=\zeta$ and
  $$
  \mathbb{Z}^N=\bigoplus_{i=1}^N\tau_i\mathbb{Z}\zeta_i,
  ~~\zeta_{i}\perp\zeta_{j}~ (i\neq j), ~~i,j=1,\ldots,N.
  $$
\end{lemma}
The proof of this key lemma is given in Appendix. Now, for any $k\in\mathbb{Z}^N$ and for some given vector $\zeta\in\mathbb{Q}^{N}\cap\mathbb{S}^{N-1}$, we decompose $k$ along $\zeta$ and $\zeta^\perp$, that is,
$$
k=(k\cdot\zeta)\zeta+k_\perp~\mbox{ with }k_\perp\in\zeta^\perp.
$$
Using Lemma \ref{change-basis}, we obtain that
$$
k\cdot\zeta=\tau_{1}p_1~\text{ and }~
k_\perp=k-(k\cdot\zeta)\zeta=\sum_{i=2}^{N}\tau_ip_i\zeta_i
$$
where $p_i\in\mathbb{Z}$ for all $1\leq i\leq N$. Recall the moving frame in Section \ref{new-moving-frame} and consider a given rational direction $\zeta\in\mathbb{Q}^{N}\cap\mathbb{S}^{N-1}$. Then, the $R-$pulsating condition \eqref{profile-periodicity} in the direction $\zeta$ actually amounts to the following periodicity condition:
\begin{align}\label{space-time-periodic-profile}
&\forall\,p_1\in\mathbb{Z},~~
\forall\,\tau\in\tau_2\mathbb{Z}\times\cdots\times\tau_N\mathbb{Z},~~
\forall\,(\xi,t,y)\in\mathbb{R}\times\mathbb{R}\times\mathbb{R}^{N-1},
\notag\\
&\varphi\left(\xi,t+\frac{k\cdot\zeta}{c},y+\hat{R}k_\perp\right)=
\varphi\left(\xi,t+\frac{p_1\tau_{1}}{c},y+\tau\right)=\varphi(\xi,t,y).
\end{align}
Likewise, $\widetilde{A}$ and $\widetilde{f}$ also satisfy the same periodicity as \eqref{space-time-periodic-profile}.
Consequently, the existence of pulsating travelling fronts of \eqref{general-equation} propagating in the direction $\zeta\in\mathbb{Q}^{N}\cap\mathbb{S}^{N-1}$ is equivalent to solving the space-time periodic problem
\begin{equation}\label{rational-direction-equation}
 \begin{dcases}
   \varphi_t-\dive_{\xi,y}\big(\widetilde{A}(\xi,t,y)\nabla_{\xi,y}\varphi\big)
   -c\varphi_{\xi}=\widetilde{f}(\xi,t,y,\varphi)~\mbox{ in }
   \mathbb{R}\times\mathbb{R}\times\mathbb{R}^{N-1},\\
   \varphi\text{ is }\frac{\tau_1}{c}\text{-periodic in }t\text{ and }
   \tau\text{-periodic in }y,\\
   \lim_{\xi\to-\infty}\varphi(\xi,t,y)=1\mbox{ and } \lim_{\xi\to+\infty}\varphi(\xi,t,y)=0,
 \end{dcases}
\end{equation}
where the limits hold uniformly with respect to $t\in\mathbb{R}$ and $y\in\mathbb{R}^{N-1}$.

Let us comment on \eqref{space-time-periodic-profile}-\eqref{rational-direction-equation}.
As underlined in Remark \ref{profile-not-periodic}, in the frame moving with speed $c$, the profile of the front propagating in any given direction $e\in\mathbb{S}^{N-1}$ is in general not periodic in time.
However, the above observation reveals a very interesting phenomenon that not only for the directions of standard coordinate vectors but also for those directions with rational coordinates, the profile of pulsating fronts is periodic in time. In the following we shall split into two steps to solve \eqref{rational-direction-equation}: the first one considers a similar space-time periodic problem in a strip and the second one is to pass to the limit in the unbounded domains. Such an approach has been widely used to show the existence of travelling fronts (see \cite{BH2002,BHR2005-2,N2009} for instances).
\subsubsection{Existence in a strip}
Here we construct a solution of \eqref{rational-direction-equation} on a bounded domain with respect to $\xi$. In the following, we fix a vector $\zeta\in\mathbb{Q}^{N}\cap\mathbb{S}^{N-1}$ and a speed $c>c^*(\zeta)$. Let $C$ be the periodicity cell defined by
$C=(0,\tau_2)\times\cdots\times(0,\tau_N)$.
We consider the functions $\overline{\varphi}$ and $\underline{\varphi}$ defined by
\begin{equation}\label{super-sub-solution}
\overline{\varphi}(\xi,t,y)=\min\{1,\psi^+(\xi,t,y)\}~\mbox{ and }
~\underline{\varphi}(\xi,t,y)=\max\{0,\psi^-(\xi,t,y)\}
\end{equation}
for all $(\xi,t,y)\in\mathbb{R}\times\mathbb{R}\times\mathbb{R}^{N-1}$, where the functions $\psi^-$ and $\psi^+$ are defined in Lemmas \ref{subsolution}-\ref{supersolution}, respectively. Let $a>0$ be given and set $\Sigma_a=(-a,a)\times\mathbb{R}\times\mathbb{R}^{N-1}$. Consider the following boundary conditions associated with Problem \eqref{rational-direction-equation}:
\begin{equation}\label{equation-strip}
 \begin{dcases}
   \varphi_t-\dive_{\xi,y}\big(\widetilde{A}(\xi,t,y)\nabla_{\xi,y}\varphi\big)
   -c\varphi_{\xi}=\widetilde{f}(\xi,t,y,\varphi)~\text{ in }\Sigma_a,
   \\
   \varphi\text{ is }\frac{\tau_1}{c}\text{-periodic in }t\text{ and }\tau\text{-periodic in }y,\\
   \forall(t,y)\in\mathbb{R}\times\mathbb{R}^{N-1},~
   \varphi(-a,t,y)=\overline{\varphi}(-a,t,y),\\
   \forall(t,y)\in\mathbb{R}\times\mathbb{R}^{N-1},~
   \varphi(a,t,y)=\underline{\varphi}(a,t,y).
 \end{dcases}
\end{equation}
Note first that there exists $a_0>0$ large enough so that for any $a\geq a_0$, $\overline{\varphi}(-a,t,y)=1$ and $\underline{\varphi}(a,t,y)=\psi^-(a,t,y)$ for all $(t,y)\in\mathbb{R}\times\mathbb{R}^{N-1}$. Then we prove the following
\begin{lemma}\label{existence-lemma}
  For any $a>a_0$, Problem \eqref{equation-strip} admits a solution $\varphi_{a}\in C^{2,1,2}(\overline{\Sigma_a})$ which satisfies
  $$
  \underline{\varphi}(\xi,t,y)
  \leq\varphi_{a}(\xi,t,y)\leq\overline{\varphi}(\xi,t,y),~~
  \forall(\xi,t,y)\in\overline{\Sigma_{a}}.
  $$
\end{lemma}

The proof of this lemma relies on arguments that are close to the ones used by Nadin in \cite[Lemma 4.3]{N2009}.

\begin{proof}
Let $\phi$ be the function defined by
$$
\phi(\xi,t,y)=\underline{\varphi}(a,t,y)\frac{\xi+a}{2a}
-\overline{\varphi}(-a,t,y)\frac{\xi-a}{2a}.
$$
This function is $\tau_1/c$-periodic in $t$ and $\tau$-periodic in $y$. Setting $v=\varphi-\phi$, Problem \eqref{equation-strip} is equivalent to
\begin{equation*}
 \begin{dcases}
   \mathcal{L}v=\widetilde{f}(\xi,t,y,v+\phi)-
   \mathcal{L}\phi~\text{ in }\Sigma_a, \\
   v\text{ is }\frac{\tau_1}{c}\text{-periodic in }t\text{ and }\tau\text{-periodic in }y,\\
 \forall(t,y)\in\mathbb{R}\times\mathbb{R}^{N-1},~
 v(\pm a,t,y)=0,
 \end{dcases}
\end{equation*}
where $\mathcal{L}$ denotes the parabolic operator given by
$$
\mathcal{L}\psi=\partial_t\psi-
\dive_{\xi,y}\big(\widetilde{A}(\xi,t,y)\nabla_{\xi,y}\psi\big)
-c\partial_{\xi}\psi.
$$
Set
\begin{equation}\label{g-v-def}
g(v)(\xi,t,y):=\widetilde{f}(\xi,t,y,v+\phi)-\mathcal{L}\phi,
\end{equation}
which is locally Lipschitz continuous with respect to $v$.

In order to use an iteration procedure, we need to prove that for some constant $M>0$, the operator $\mathcal{L}+M$ is invertible.
To that aim we take some $g\in C^0(\overline{\Sigma_a})$ such that $g$ is $\tau_{1}/c$-periodic in $t$ and $\tau$-periodic in $y$, and furthermore we define the set
$$
\Gamma^0_{\rm per}=\{v\in L^2\left((-a,a)\times C\right)\mbox{ such that } v\mbox{ is }\tau\mbox{-periodic in }y\}
$$
and endow it with the $L^2$-norm. Then it is a Banach space. Now, for all $v_0\in \Gamma^0_{\rm per}$, we define $(\xi,t,y)\mapsto v(\xi,t,y)\in C^1((0,+\infty);\Gamma^0_{\rm per})$ as the solution of
 \begin{equation*}
 \begin{dcases}
   \mathcal{L}v+Mv=g(\xi,t,y)&\mbox{ in }(-a,a)\times(0,\infty)
  \times\mathbb{R}^{N-1},\\
   v(-a,t,y)=v(a,t,y)=0&\mbox{ in }(0,\infty)
  \times\mathbb{R}^{N-1},\\
   v(\xi,0,y)=v_{0}(\xi,y)&\mbox{ in }(-a,a)\times\mathbb{R}^{N-1},
 \end{dcases}
\end{equation*}
and we investigate the Poincar\'{e} map
 \begin{equation*}
 \begin{dcases}
   \mathcal{T}:\Gamma^0_{\rm per}\to\Gamma^0_{\rm per}, \\
   v_{0}\mapsto v\left(\frac{\tau_1}{c}\right).
 \end{dcases}
\end{equation*}
Take $v_{01}, v_{02}\in\Gamma^0_{\rm per}$ and set $V(\xi,t,y)=(v_1(\xi,t,y)-v_2(\xi,t,y))e^{Mt}$. Then, $V$ satisfies
$$
V_t-\dive_{\xi,y}\big(\widetilde{A}(\xi,t,y)\nabla_{\xi,y}V\big)
-cV_{\xi}=0.
$$
Multiplying this equation by $V$ and integrating by parts over $(-a,a)\times(0,\frac{\tau_1}{c})
\times C$, we obtain that
\begin{align*}
&\frac{1}{2}\int\limits_{(-a,a)\times C}\left(V^{2}
\left(\xi,\frac{\tau_1}{c},y\right)-V^{2}(\xi,0,y)\right)
\mathrm{d}\xi\mathrm{d}y
=-\int_{-a}^a\int_{0}^{\frac{\tau_1}{c}}\int_C
\nabla_{\xi,y}V\widetilde{A}(\xi,t,y)\nabla_{\xi,y}V.
\end{align*}
Using the uniform elliptic conditions for $\widetilde{A}$, one gets
$$
\left\|v_2\left(\frac{\tau_1}{c}\right)-v_1\left(\frac{\tau_1}{c}\right)
\right\|_{L^2\left((-a,a)\times C\right)}\leq e^{-M\frac{\tau_1}{c}}\|v_2(0)-v_1(0)\|_{L^2\left((-a,a)\times C\right)}.
$$
Since $M$ and $\tau_1/c$ are positive, the map $\mathcal{T}$ is a contraction from $\Gamma^0_{\rm per}$ into itself. It then follows from Banach fixed point theorem that $\mathcal{T}$ admits a unique fixed point in $\Gamma^0_{\rm per}$. In other words, there exists a space-time periodic function $v$ such that $\mathcal{L}v+Mv=g$. Furthermore, from standard parabolic estimates and using the periodicity in $t$, we can obtain the point-wise boundary conditions:
$$
v(-a,t,y)=v(a,t,y)=0,~~\forall(t,y)\in\mathbb{R}\times\mathbb{R}^{N-1}.
$$

One can now carry out the iteration procedure. Recalling that the function $g(v)$ defined in \eqref{g-v-def} is locally Lipschitz continuous with respect to $v$, then we can fix a number $M>0$ large enough so that the mapping
\begin{align}\label{iteration}
s\in\mathbb{R}_+\mapsto\widetilde{f}(\xi,t,y,s)+Ms\mbox{ is increasing for all }
(\xi,t,y)\in\Sigma_a.
\end{align}
We inductively define the sequence $\{\varphi_n\}_{n\in\mathbb{N}_0}$ by
\begin{equation*}
 \begin{dcases}
 \varphi_0(\xi,t,y)=\underline{\varphi}(\xi,t,y),\\
   \mathcal{L}\varphi_{n+1}+M\varphi_{n+1}=
   \widetilde{f}(\xi,t,y,\varphi_{n})+M\varphi_{n}~\mbox{ in }\Sigma_a,\\
   \forall(t,y)\in\mathbb{R}\times\mathbb{R}^{N-1},~
   \varphi_{n+1}(-a,t,y)=\overline{\varphi}(-a,t,y),\\
   \forall(t,y)\in\mathbb{R}\times\mathbb{R}^{N-1},~
   \varphi_{n+1}(a,t,y)=\underline{\varphi}(a,t,y).
 \end{dcases}
\end{equation*}
For $n=0$, due to $0\leq\underline{\varphi}\leq1$, we have
\begin{equation*}
 \begin{dcases}
 \mathcal{L}\varphi_1+M\varphi_1=
 \widetilde{f}(\xi,t,y,\underline{\varphi})+M\underline{\varphi}\geq0
 ~\mbox{ in }\Sigma_a,\\
 \varphi_1\text{ is }\frac{\tau_1}{c}\text{-periodic in }t\text{ and }\tau\text{-periodic in }y,\\
  \forall(t,y)\in\mathbb{R}\times\mathbb{R}^{N-1},~
 \varphi_1(-a,t,y)=\overline{\varphi}(-a,t,y)>0,\\
 \forall(t,y)\in\mathbb{R}\times\mathbb{R}^{N-1},~
 \varphi_1(a,t,y)=\psi^-(a,t,y)>0.
 \end{dcases}
\end{equation*}
Since $M>0$, the weak parabolic maximum principle and the periodicity conditions yield $\varphi_1\geq0$ in $\overline{\Sigma_a}$. Indeed, if $\min_{\overline{\Sigma_a}}\varphi_1<0$, then the point $P_0:=(\xi_0,t_0,y_0)$ where $\varphi_1$ attains the minimum lies in $\Sigma_a$ due to $\varphi_1(\pm a,\cdot,\cdot)>0$. Since $\varphi_1$ is periodic in $t$, we have
$$\partial_t\varphi_1(P_0)=0,~
\nabla_{\xi,y}\varphi_1(P_0)=0_{\mathbb{R}^N}\mbox{ and }
D^2_{\xi,y}\varphi_1(P_0)\mbox{ is nonnegative-definite}.
$$
Using the uniform elliptic condition and symmetry of $\widetilde{A}$, one can further verify that
$$
 \dive_{\xi,y}\big(\widetilde{A}\nabla_{\xi,y}\varphi_1\big)(P_0)\geq0.
$$
Thanks to $M>0$, one gets $(\mathcal{L}+M)\varphi_1|_{P_0}<0$, whence a contradiction has been achieved. Using the periodicity of $\varphi_1$ in $t$ and $\varphi_1(\pm a,\cdot,\cdot)>0$, it then follows from the strong parabolic maximum principle that $\varphi_1>0$ in $\overline{\Sigma_a}$. Next, let us show that $\varphi_1\geq\varphi_0$. By Lemma \ref{subsolution}, we have
\begin{equation*}
 \begin{dcases}
\mathcal{L}(\varphi_1-\psi^-)+M(\varphi_1-\psi^-)=
\widetilde{f}(\xi,t,y,\psi^-)-\mathcal{L}\psi^-\geq0
&\mbox{ in }\Omega^+\cap\Sigma_a,\\
(\varphi_1-\psi^-)(\xi,t,y)\geq0
&\mbox{ on }\partial(\Omega^+\cap\Sigma_a),
 \end{dcases}
\end{equation*}
where
$$
\partial(\Omega^+\cap\Sigma_a)=
\left(\Omega^0\cap\overline{\Sigma_a}\right)\cup
\left(\Omega^+\cap\partial\Sigma_a\right)
$$
with
$$
\Omega^+=\{(\xi,t,y):\psi^->0\}~\text{ and }~
\Omega^0=\{(\xi,t,y):\psi^-\equiv0\}.
$$
Since $M>0$ and $(\varphi_1-\psi^-)$ is periodic in $t$ and $y$, it follows from the weak maximum principle that $\varphi_1\geq\psi^-$ in $\overline{\Omega^+}\cap\overline{\Sigma_a}$. Furthermore, as $0$ is a solution of Eq. \eqref{profile-equation} and $\varphi_1>0$ in $\overline{\Sigma_a}$, we further get
$$
\varphi_1\geq\underline{\varphi}=\max\{0,\psi^-\}
~\mbox{ in }\overline{\Sigma_a}.
$$
Now assume that $\underline{\varphi}\leq\varphi_1\leq\cdots\leq\varphi_{n-1}\leq\varphi_n$ in $\overline{\Sigma_a}$. Due to \eqref{iteration}, one has
$$
\mathcal{L}(\varphi_{n+1}-\varphi_n)+M(\varphi_{n+1}-\varphi_n)\geq0
~\mbox{ in }\Sigma_a.
$$
Hence the weak maximum principle ensures that $\varphi_{n+1}\geq\varphi_n$ in $\overline{\Sigma_a}$. Then we conclude that
$$
\underline{\varphi}=\varphi_0\leq\varphi_1\leq\cdots\leq\varphi_n\leq
\varphi_{n+1}\leq\cdots
\mbox{ in }\overline{\Sigma_a}.
$$
We proceed to show that $\varphi_n\leq\overline{\varphi}=\min\{1,\psi^+\}$ in $\overline{\Sigma_a}$ for all $n\in\mathbb{N}_0$. Notice first that $\varphi_0\leq\overline{\varphi}$. Assume for induction that $\varphi_n\leq\overline{\varphi}$ in $\overline{\Sigma_a}$. Due to \eqref{iteration} and by Lemma \ref{supersolution}, we obtain that
\begin{equation*}
 \begin{dcases}
\mathcal{L}(\psi^+-\varphi_{n+1})+M(\psi^+-\varphi_{n+1})\geq
\mathcal{L}\psi^+-\widetilde{f}(\xi,t,y,\psi^+)\geq0
~\mbox{ in }\Sigma_a,\\
\forall(t,y)\in\mathbb{R}\times\mathbb{R}^{N-1},~
(\psi^+-\varphi_{n+1})(\pm a,t,y)\geq0.
 \end{dcases}
\end{equation*}
Since $(\psi^+-\varphi_{n+1})$ is periodic in $t$ and $y$, it follows from the weak maximum principle that $\psi^+\geq\varphi_{n+1}$ in $\overline{\Sigma_a}$. Similarly, as $1$ is a solution of Eq. \eqref{profile-equation}, we can prove that $1\geq\varphi_{n+1}$ in $\overline{\Sigma_a}$. This indicates that
$\varphi_{n+1}\leq\overline{\varphi}=\min\{1,\psi^+\}
\mbox{ in }\overline{\Sigma_a}$.
Therefore, we get a nondecreasing sequence with respect to $n$ between $\underline{\varphi}$ and $\overline{\varphi}$, that is,
$$
\underline{\varphi}\leq\varphi_1\leq\cdots\leq\varphi_n\leq
\varphi_{n+1}\leq\cdots\leq\overline{\varphi}
~~\mbox{ in }\overline{\Sigma_a}.
$$
From the standard parabolic $L^p$-estimates and using the periodicity conditions, this sequence converges to some function $\varphi_a$ which is a strong solution of Problem \eqref{equation-strip}.
Furthermore, from the parabolic regularity theory up to the boundary, the solution $\varphi_a$ belongs to $C^{2,1,2}(\overline{\Sigma_a})$. This ends the proof of Lemma \ref{existence-lemma}.
\end{proof}
\subsubsection{Passage to the limit in the unbounded domains}
Let $\{a_n\}_{n\in\mathbb{N}}$ be a sequence such that $a_n\to+\infty$ as $n\to+\infty$ and let $\varphi_{a_n}$ be a solution of \eqref{equation-strip} with $a=a_n$. From standard parabolic estimates and Sobolev's injections, the sequence $\{\varphi_{a_n}\}$ converges (up to the extraction of some subsequence), for all $\beta\in[0,\alpha)$ in $C_{\rm loc}^{2+\beta,1+\beta/2,2+\beta}(\mathbb{R}\times\mathbb{R}\times
\mathbb{R}^{N-1})$,  to a function $\varphi$ which satisfies
\begin{equation*}
 \begin{dcases}
   \varphi_t-\dive_{\xi,y}\big(\widetilde{A}(\xi,t,y)\nabla_{\xi,y}
   \varphi\big)-c\varphi_{\xi}=\widetilde{f}(\xi,t,y,\varphi)~
   \mbox{ in }\mathbb{R}\times\mathbb{R}\times\mathbb{R}^{N-1},\\
   \varphi\text{ is }\frac{\tau_1}{c}\text{-periodic in }t\text{ and }\tau\text{-periodic in }y,\\
   \underline{\varphi}(\xi,t,y)\leq\varphi(\xi,t,y)\leq
   \overline{\varphi}(\xi,t,y),~~
   \forall(\xi,t,y)\in\mathbb{R}\times\mathbb{R}\times\mathbb{R}^{N-1}.
 \end{dcases}
\end{equation*}
\begin{proposition}\label{limits}
The solution $\varphi$ satisfies the following behaviour at $\xi=\pm\infty$
\begin{equation*}
 \begin{dcases}
  \lim_{\xi\to+\infty}\varphi(\xi,t,y)=0,\\
  \lim_{\xi\to-\infty}\varphi(\xi,t,y)=1,
 \end{dcases}
 \mbox{ uniformly for }(t,y)\in\mathbb{R}\times\mathbb{R}^{N-1},
\end{equation*}
as well as the following property
\begin{equation}\label{monotonicity-profile}
\varphi(\xi,t,y)\leq\varphi(\xi-ct^\prime,t+t^\prime,y),
~~\forall(\xi,t,y)\in\mathbb{R}\times\mathbb{R}\times\mathbb{R}^{N-1},
~\forall t^\prime>0.
\end{equation}
\end{proposition}
\begin{proof}
Note that the functions $\underline{\varphi}$ and $\overline{\varphi}$ defined in \eqref{super-sub-solution} converge to $0$ as $\xi\to+\infty$. Using the periodicity in $t$ and $y$, it immediately follows that $\varphi(\xi,t,y)\to0$ as $\xi\to+\infty$ uniformly for $(t,y)\in\mathbb{R}\times\mathbb{R}^{N-1}$. Furthermore, $\varphi$ satisfies the following exponential decay
\begin{equation}\label{decay-profile}
\varphi(\xi,t,y)\sim e^{-\lambda_c\xi}\Phi_{\lambda_c}(R(\xi+ct,y)^T)
\text{ as $\xi\to+\infty$ uniformly for $(t,y)\in\mathbb{R}^{N}$},
\end{equation}
where the function $\Phi_{\lambda_{c}}$ is the unique principal eigenfunction of Problem \eqref{periodic-eigenvalue-problem} in the sense of \eqref{normalization} corresponding to $\lambda=\lambda_{c}>0$.

It remains to prove the left limit and the property \eqref{monotonicity-profile}. Let us mention that the following proof is inspired from \cite{BMJ-TAMS}. To prove the limit of $\varphi$ at $-\infty$, we first show that
\begin{equation}\label{left-limit}
  \liminf_{\xi\to-\infty,\,y\in\overline{C}}\varphi(\xi,0,y)>0.
\end{equation}
Assume by contradiction that there exists a sequence $\{(\xi_n,y_n)\}\subset\mathbb{R}\times\overline{C}$ with $\xi_n\to-\infty$ as $n\to\infty$ such that $\varphi(\xi_n,0,y_n)\to0$ as $n\to\infty$. Then, for any given $\epsilon>0$ small, there exists $N_0>0$ sufficiently large such that $\varphi(\xi_n,0,y_n)\leq\epsilon$ for all $n\geq N_0$. From Harnack's inequality and standard parabolic estimates, we obtain that there exists some constant $M>0$ such that
\begin{align}\label{consequence-Harnack}
 \varphi(\xi_n,t,y)\leq M\epsilon,~~
 \norm{\nabla_{\xi,y}\varphi(\xi_n,t,y)}\leq M\epsilon,~~
 \forall(t,y)\in\left[0,\frac{\tau_1}{c}\right]\times\overline{C}.
\end{align}
Now, taking some $l>0$ and integrating the equation over $(\xi_n,l)\times(0,\tau_1/c)\times C$, one gets
\begin{align*}
  &\int_{\xi_n}^l\int_C\left[\varphi\left(\xi,\frac{\tau_1}{c},y\right)-\varphi
  (\xi,0,y)\right]\mathrm{d}\xi\mathrm{d}y
-\int_{0}^{\frac{\tau_1}{c}}\int_C\left(e_1\widetilde{A}(\xi,t,y)\nabla_{\xi,y}
\varphi\right)\Big|^l_{\xi_n}\mathrm{d}t\mathrm{d}y\\
-&\int_{\xi_n}^l\int_{0}^{\frac{\tau_1}{c}}\int_C\dive_{y}\big(B\widetilde{A}(\xi,t,y)
\nabla_{\xi,y}\varphi\big)\mathrm{d}\xi\mathrm{d}t\mathrm{d}y
-c\int_{0}^{\frac{\tau_1}{c}}\int_C\varphi(\xi,t,y)\big|^l_{\xi_n}
\mathrm{d}t\mathrm{d}y\\
&=\int_{\xi_n}^l\int_{0}^{\frac{\tau_1}{c}}\int_C\widetilde{f}(\xi,t,y,\varphi)
\mathrm{d}\xi\mathrm{d}t\mathrm{d}y
\end{align*}
where $B:=(e_2,\ldots,e_N)^T$ is a $(N-1)\times N$ matrix. Using the periodicity conditions in $t$ and $y$, we obtain that $\mathscr{L}(\xi_n,l)\equiv\mathscr{R}(\xi_n,l)$ with
\begin{align*}
\mathscr{L}(\xi_n,l):&=-\int_{0}^{\frac{\tau_1}{c}}\int_C\left(e_1\widetilde{A}\nabla_{\xi,y}
\varphi\right)\Big|^l_{\xi_n}\mathrm{d}t\mathrm{d}y
-c\int_{0}^{\frac{\tau_1}{c}}\int_C\left[\varphi(l,t,y)-
  \varphi(\xi_n,t,y)\right]\mathrm{d}t\mathrm{d}y,\\
\mathscr{R}(\xi_n,l):&=\int_{\xi_n}^l\int_{0}^{\frac{\tau_1}{c}}\int_C\widetilde{f}(\xi,t,y,\varphi)
  \mathrm{d}\xi\mathrm{d}t\mathrm{d}y.
\end{align*}
Next we prove that the equality $\mathscr{L}(\xi_n,l)\equiv\mathscr{R}(\xi_n,l)$ does no longer hold when $n\to\infty$ and $l\to+\infty$. On the one hand, due to $0\not\equiv\underline{\varphi}\leq\varphi\leq\overline{\varphi}\not\equiv1$ and since $\varphi$ is periodic in $t$, the strong maximum principle yields $0<\varphi<1$ for any $(\xi,t,y)\in\mathbb{R}\times\mathbb{R}\times\mathbb{R}^{N-1}$, whence $\widetilde{f}(\cdot,\cdot,\cdot,\varphi)>0$ by our hypothesis \eqref{reaction-hypothesis}. Letting $n\to\infty$ and $l\to+\infty$, we obtain that
$$
\mathscr{R}(-\infty,+\infty):=\int_{-\infty}^{+\infty}
\int_{0}^{\frac{\tau_1}{c}}\int_C\widetilde{f}
(\xi,t,y,\varphi)\mathrm{d}\xi\mathrm{d}t\mathrm{d}y>0.
$$
Then there exist $n_0$ and $l_0$ large enough such that
\begin{align}\label{RightHand}
\mathscr{R}(\xi_n,l)\geq \frac{\mathscr{R}(-\infty,+\infty)}{2}>0,
~~\forall n>n_0,~\forall l>l_0.
\end{align}
On the other hand, since $\varphi(l,\cdot,\cdot)\to0$ as $l\to+\infty$, the standard parabolic estimates imply that $\nabla_{\xi,y}\varphi(l,t,y)\to0_{\mathbb{R}^N}$ as $l\to+\infty$ uniformly in $(t,y)\in[0,\tau_1/c]\times\overline{C}$, and that $\varphi(l,t,y)$ and $\|\nabla\varphi(l,t,y)\|$ are bounded by $\epsilon$ for $l$ large enough. By \eqref{consequence-Harnack} and from the uniform elliptic condition \eqref{elliptic-condition}, we have
$$
\mathscr{L}(\xi_n,+\infty):=\int_{0}^{\frac{\tau_1}{c}}\int_C \left(e_1\widetilde{A}\nabla_{\xi,y}\varphi(\xi_n,t,y)+
c\varphi(\xi_n,t,y)\right)\mathrm{d}t\mathrm{d}y
\leq M|C|\left(\frac{\tau_1}{c}\Gamma+\tau_1\right)\epsilon.
$$
Then we deduce that $\mathscr{L}(\xi_n,+\infty)\to0$ as $n\to\infty$. Therefore, there exists some constant $M^{\prime}>0$ such that
\begin{align}\label{LeftHand}
\mathscr{L}(\xi_n,l)\leq M^{\prime}\epsilon,~~\forall n>n_0,~\forall l>l_0.
\end{align}
However, \eqref{RightHand} implies that $\mathscr{R}(\xi_n,l)$ is uniformly bounded from below by a positive constant. Thus, from \eqref{LeftHand} and using the fact that $\mathscr{L}(\xi_n,l)\equiv\mathscr{R}(\xi_n,l)$, a contradiction has been achieved.

One can now finish the proof of $\varphi(-\infty,\cdot,\cdot)=1$.
Note first that \eqref{left-limit} implies that there exist $\alpha_1>0$ and $K>0$ large enough such that $\varphi(\xi,0,y)\geq\alpha_1>0$ for any $(\xi,y)\in(-\infty,-K]\times\overline{C}$.
Since $\varphi$ is positive, then $\varphi(-K,t,y)$ is also bounded from below by a positive constant $\alpha_2$ for any $(t,y)\in[0,\tau_1/c]\times\overline{C}$. Set $\alpha:=\min\{\alpha_1,\alpha_2,1\}$ and consider the initial/boundary value problem posed for  $(\xi,t,y)\in(-\infty,-K]\times[0,\tau_1/c]\times\overline{C}$:
\begin{equation*}
 \begin{dcases}
   \varphi_t-\dive_{\xi,y}\big(\widetilde{A}(\xi,t,y)\nabla_{\xi,y} \varphi\big)-c\varphi_{\xi}=
   \widetilde{f}(\xi,t,y,\varphi)\geq0,\\
   \varphi(-K,t,y)\geq\alpha,\\
 \varphi(\xi,0,y)\geq\alpha.
 \end{dcases}
\end{equation*}
It then follows from the maximum principle that  $\varphi(\xi,t,y)\geq\alpha>0$ for all $(\xi,t,y)\in(-\infty,-K]\times[0,\tau_1/c]\times\overline{C}$. Next, set $v:=1-\varphi$ and $f(x,u):=(1-u)h(x,u)$. Then the function $v$ satisfies the equation
$$
v_t-\dive_{\xi,y}\big(\widetilde{A}(\xi,t,y)\nabla_{\xi,y} v\big)
-cv_{\xi}+\widetilde{h}(\xi,t,y,\varphi)v=0
$$
wherein we have set $\widetilde{h}(\xi,t,y,\varphi):=h(R(\xi+ct,y)^T,u)$. From our hypothesis \eqref{reaction-hypothesis}, we know that there exists some constant $h_0>0$ such that $h(X,u)\geq h_0>0$ for all $(X,u)\in\mathbb{T}^N\times[\alpha,1]$. Now, we consider the following function
$$
\widehat{v}(\xi,t,y):=\|v(\cdot,0,\cdot)\|_{\infty}e^{-h_0 t}
+\omega e^{\mu\xi},
~~\forall(\xi,t,y)\in(-\infty,-K]\times[0,+\infty)\times\overline{C}
$$
with a constant $\omega>0$ and we choose $\mu>0$ small enough so that $\mu^2+c\mu\leq h_0$.
Then, we obtain that
\begin{equation*}
 \begin{dcases}
   (\widehat{v}-v)_t-\dive_{\xi,y}\big(\widetilde{A}(\xi,t,y)\nabla_{\xi,y} (\widehat{v}-v)\big)-c(\widehat{v}-v)_{\xi}+\widetilde{h}(\xi,t,y,\varphi)
  (\widehat{v}-v)\geq0,\\
   (\widehat{v}-v)(-K,t,y)\geq\omega e^{-\mu K}-1,\\
 (\widehat{v}-v)(\xi,0,y)\geq0
 \end{dcases}
\end{equation*}
for all $(\xi,t,y)\in(-\infty,-K]\times[0,+\infty)\times\overline{C}$. Therefore, if $\omega\geq e^{\mu K}$, it then follows from the maximum principle that $v(\xi,t,y)\leq\widehat{v}(\xi,t,y)$ for all $(\xi,t,y)\in(-\infty,-K]\times[0,+\infty)\times\overline{C}$. Since $v$ is $\tau_1/c$-periodic in $t$ and $\tau$-periodic in $y$, by letting $t=t+n\tau_1/c$ and then passing to the limit $n\to+\infty$, we obtain that
$$
1-\varphi(\xi,t,y)=v(\xi,t,y)\leq \omega e^{\mu\xi},~~\forall\xi\leq-K, ~\forall(t,y)\in\mathbb{R}\times\mathbb{R}^{N-1},
$$
which concludes that $\lim_{\xi\to-\infty}\varphi(\xi,t,y)=1$ uniformly for $t\in\mathbb{R}$ and $y\in\mathbb{R}^{N-1}$.

Lastly, let us turn to the proof of \eqref{monotonicity-profile}.
To this end, we fix any $t^\prime>0$ and set
$$
w(\xi,t,y):=\varphi(\xi-ct^\prime,t+t^\prime,y)-\varphi(\xi,t,y).
$$
The function $w$ is bounded, $\tau_1/c$-periodic in $t$ and $\tau$-periodic in $y$.
Moreover, there exists a bounded and nonnegative function $b(\xi,t,y;t^\prime)$ such that
$$
w_t-\dive_{\xi,y}\big(\widetilde{A}(\xi,t,y)\nabla_{\xi,y}w\big)
-cw_{\xi}=\widetilde{f}_{\varphi}(\xi,t,y,b(\xi,t,y;t^\prime))w.
$$
Due to \eqref{decay-profile}, one has
\begin{equation}\label{decay_w0}
w(\xi,0,y)\sim (e^{c\lambda_c t^\prime}-1)
\Phi_{\lambda_{c}}(R(\xi,y)^T)e^{-\lambda_{c}\xi}>0
\text{ as }\xi\to\infty
\text{ uniformly for }y\in\mathbb{R}^{N-1}.
\end{equation}
Next we consider the following function
\begin{equation}\label{def_W}
W(\xi,t,y):=e^{-r_\delta t}\Phi_{\lambda_{c}+\delta}(R(\xi+ct,y)^T)
e^{-(\lambda_{c}+\delta)\xi},
\end{equation}
where the parameter $\delta$ is chosen as in Lemma \ref{subsolution} so that $r_\delta$ defined by \eqref{r_delta} is positive.
From the proof of Lemma \ref{subsolution}, we know that $W$ satisfies
$$
W_t-\dive_{\xi,y}\big(\widetilde{A}(\xi,t,y)\nabla_{\xi,y}W\big)
-cW_{\xi}=\widetilde{f}_{\varphi}(\xi,t,y,0)W.
$$
Observe from \eqref{decay_w0} and \eqref{def_W} that there exists a constant $\eta>0$ such that
$$
w(\xi,0,y)+\eta W(\xi,0,y)\geq0,~~\forall(\xi,y)\in\mathbb{R}\times
\mathbb{R}^{N-1}.
$$
Furthermore, one also has
\begin{align*}
(w+\eta W)_t&-\dive_{\xi,y}\big(\widetilde{A}\nabla_{\xi,y}
(w+\eta W)\big)-c(w+\eta W)_{\xi}-
\widetilde{f}_{\varphi}(\xi,t,y,b)(w+\eta W)\\
&=\left(\widetilde{f}_{\varphi}(\xi,t,y,0)-\widetilde{f}_{\varphi}
(\xi,t,y,b(\xi,t,y;t^\prime))\right)\eta W\geq0
\end{align*}
due to our hypotheses \eqref{reaction-hypothesis}. Consequently, the weak maximum principle ensures that
$$
(w+\eta W)(\xi,t,y)\geq0,~~\forall(\xi,t,y)\in\mathbb{R}
\times\mathbb{R}_+\times\mathbb{R}^{N-1}.
$$
Choose $s\in[0,\tau_1/c]$ and take $t=s+n\tau_1/c$.
Since the function $w$ is periodic in $t$, passing to the limit $n\to\infty$, we conclude from \eqref{def_W} that $w$ is nonnegative in $\mathbb{R}
\times\mathbb{R}\times\mathbb{R}^{N-1}$, whence \eqref{monotonicity-profile} follows.
This ends the proof of Proposition \ref{limits}.
\end{proof}
\subsection{Rational approximation to any direction of propagation}
This subsection is devoted to the proof of the existence of pulsating travelling fronts of \eqref{general-equation} propagating in an arbitrarily given direction $e\in\mathbb{S}^{N-1}$. Note first that in Subsection \ref{existence-rational-direction} we have shown that for each vector $\zeta\in\mathbb{Q}^N\cap\mathbb{S}^{N-1}$, there exists a periodic (in the last two variables $(t,y)$) solution $\varphi(\xi,t,y;\zeta)$ of Problem \eqref{rational-direction-equation}. Coming back to the original variables, it equivalently says that Problem \eqref{general-equation} admits a pulsating travelling front $u(t,x;\zeta)$ propagating in the direction $\zeta$ with the speed $c>c^*(\zeta)$ which satisfies
\begin{equation}\label{rational-approximation-pulsation}
 \begin{dcases}
  u\left(t+\frac{k \cdot \zeta}{c},x\right)=u(t,x-k),~~\forall
   (t,x,k)\in\mathbb{R}\times\mathbb{R}^{N}\times\mathbb{Z}^{N},\\
   \lim\limits_{r\to+\infty}u(t,r\zeta+y)=0~~\mbox{and}~~
   \lim\limits_{r\to-\infty}u(t,r\zeta+y)=1,
 \end{dcases}
\end{equation}
where the limits hold locally uniformly for $t\in\mathbb{R}$ and uniformly with respect to $y\in \zeta^\bot$. Moreover, one has
\begin{align}\label{rational-approximation-boundedness}
  0\leq\underline{u}(t,x;\zeta)\leq u(t,x;\zeta)\leq\overline{u}(t,x;\zeta)\leq1,~~\forall(t,x)
\in\mathbb{R}\times\mathbb{R}^{N-1},
\end{align}
where the functions $\underline{u}$ and $\overline{u}$ are respectively of the form:
\begin{gather*}
\underline{u}(t,x;\zeta)=\max\left\{0,\Phi_{\lambda_c(\zeta),\zeta}(x)
e^{-\lambda_c(\zeta)(x\cdot\zeta-ct)}-K\Phi_{\lambda_c(\zeta)+\delta,\zeta}(x)
e^{-(\lambda_c(\zeta)+\delta)(x\cdot\zeta-ct)}\right\},\\
\overline{u}(t,x;\zeta)=\min\left\{1,\Phi_{\lambda_c(\zeta),\zeta}(x)
e^{-\lambda_c(\zeta)(x\cdot\zeta-ct)}\right\}.
\end{gather*}
Furthermore, the pulsating travelling front is increasing in time, namely
\begin{proposition}\label{monotonicity-front-time}
   The function $u(t,x;\zeta)$ satisfies $\partial_t u(t,x;\zeta)>0$ for all $(t,x)\in\mathbb{R}\times\mathbb{R}^N$.
\end{proposition}
\begin{proof}
Recalling the definition of the profile $\varphi(\xi,t,y)$ as in \eqref{def-phi} and by \eqref{monotonicity-profile}, one has
$$
u(t,R(x\cdot\zeta,y)^T)=\varphi(x\cdot\zeta-ct,t,y)\leq
\varphi(x\cdot\zeta-ct-ct^\prime,t+t^\prime,y)
=u(t+t^\prime,R(x\cdot\zeta,y)^T),
$$
namely
$u(t,x;\zeta)\leq u(t+t^\prime,x;\zeta)$ for all $(t,x)
\in\mathbb{R}\times\mathbb{R}^N$ and for any $t^\prime>0$.
We conclude that $u(t,x;\zeta)$ is nondecreasing in $t$. Furthermore, applying the strong maximum principle to the equation satisfied by $\partial_t u$ yields $\partial_t u(t,x;\zeta)>0$ for all $(t,x)\in\mathbb{R}\times\mathbb{R}^N$.
\end{proof}
We can also refer to Hamel \cite{Hamel2008} (see also \cite{BH2002,BHR2005-2}) for the other proof of the monotonicity of monostable pulsating fronts with respect to time, which is more general and robust.

Now, in order to prove the existence of pulsating travelling fronts of Problem \eqref{general-equation} in any given directions in $\mathbb{S}^{N-1}$, we need the following well-known result:
\begin{lemma}[See \cite{DG2021}]
  The set $\mathbb{Q}^N\cap\mathbb{S}^{N-1}$ is dense in $\mathbb{S}^{N-1}$.
\end{lemma}
A similar result can be found in \cite[Proposition 4.1]{BHN2008}. Thanks to this lemma, for any given $e\in\mathbb{S}^{N-1}$, there exists a sequence $\{\zeta_m\}_{m\in\mathbb{N}}\subset\mathbb{Q}^N\cap\mathbb{S}^{N-1}$ such that $\zeta_m\to e$ as $m\to\infty$. As stated above, we have proved that for each vector $\zeta_m\in\mathbb{Q}^N\cap\mathbb{S}^{N-1}$ and for each speed $c>c^*(\zeta_m)$, there exists a pulsating travelling front $u(t,x;\zeta_m)$ of \eqref{general-equation} satisfying \eqref{rational-approximation-pulsation} and \eqref{rational-approximation-boundedness}.
Consider now a sequence
$u_m(t,x):=u(t,x;\zeta_m)$ which satisfies
$$
\partial_t u_m(t,x)>0\mbox{ and }
\underline{u}(t,x;\zeta_m)\leq u_m(t,x)
\leq\overline{u}(t,x;\zeta_m),~\forall
(t,x)\in\mathbb{R}\times\mathbb{R}^N,
~\forall m\in\mathbb{N}.
$$
Since $\{u_m\}_{m\in\mathbb{N}}$ is uniformly bounded with respect to $m$, passing to the limit $m\to\infty$ and from standard parabolic estimates, we obtain that the sequence $u_m$ converges (up to the extraction of some subsequence) locally uniformly with respect to $(t,x)\in\mathbb{R}\times\mathbb{R}^N$, along with its derivative, to some function $u_\infty$ which is an entire classical solution of \eqref{general-equation}. Moreover, the convergence above also implies that $u_\infty$ satisfies
\begin{equation*}
  \begin{dcases}
    u_{\infty}\left(t+\frac{k\cdot e}{c},x\right)=u_{\infty}(t,x-k),&\forall
   (t,x,k)\in\mathbb{R}\times\mathbb{R}^{N}\times\mathbb{Z}^{N}, \\
    \partial_t u_\infty(t,x)\geq0,&\forall
   (t,x)\in\mathbb{R}\times\mathbb{R}^{N}.
  \end{dcases}
\end{equation*}
By Propositions \ref{principal-eigenelements}-\ref{minimal-speed}, one can easily check that $\underline{u}$ and $\overline{u}$ depend continuously on the direction of propagation with respect to the uniform topology, and $c^*(\zeta_m)\to c^*(e)$ as $m\to\infty$. Therefore, for any $c>c^*(e)>0$, we have
\begin{align}\label{any-direction-solution-boundedness}
 0\leq\underline{u}(t,x;e)\leq u_{\infty}(t,x)\leq\overline{u}(t,x;e)\leq1,
~~\forall(t,x)\in\mathbb{R}\times\mathbb{R}^N.
\end{align}

In order to obtain our result, we have to prove that $u_\infty(t,re+y)\to0$ as $r\to+\infty$ and $u_\infty(t,re+y)\to1$ as $r\to-\infty$ locally uniformly in $t\in\mathbb{R}$ and uniformly in $y\in e^\perp$. From the pulsating properties of $u_\infty$ and the positivity of speed $c$, it is equivalent to prove that $u_\infty(t,x)\to0$ as $t\to-\infty$ and $u_\infty(t,x)\to1$ as $t\to+\infty$ locally uniformly in $x$. The former is immediately deduced from \eqref{any-direction-solution-boundedness} and the fact that $\overline{u}\not\equiv1$ is nondecreasing with respect to $t$. Now, using the pulsating properties of $u_\infty$ and estimate \eqref{any-direction-solution-boundedness}, as well as Proposition \ref{monotonicity-front-time}, we obtain that $u_\infty$ converges in $C^2(\mathbb{T}^N)$ as $t\to+\infty$ to a $\mathbb{Z}^N$-periodic function $p$ which satisfies
\begin{equation*}
  \begin{dcases}
  \dive(A(x)\nabla p)+f(x,p)=0~\text{  in }\mathbb{T}^N,\\
  0\leq p(x)\leq1, ~\forall x\in\mathbb{T}^N.
  \end{dcases}
\end{equation*}
From our hypotheses \eqref{reaction-hypothesis} on $f$, it can be only $0$ or $1$. If $p\equiv0$, then $u_{\infty}(t,x)=0$ for all $(t,x)\in\mathbb{R}\times\mathbb{R}^N$, which is impossible due to the estimate \eqref{any-direction-solution-boundedness} and $\underline{u}(t,x;e)\not\equiv0$. This shows that $u_\infty(t,x)\to1$ as $t\to+\infty$ locally uniformly in $x$. Finally, the strong maximum principle yields that $u_\infty$ is increasing in $t$.
\subsection{Existence of a pulsating front with minimal speed}
In this subsection we prove the existence of pulsating travelling front with critical speed. To do so, let $\{c_n\}_{n\in\mathbb{N}}$ be a sequence of speeds such that $c_n>c^*(e)>0$ for each $n\in\mathbb{N}$ and $c_n\searrow c^*(e)$ as $n\to\infty$. For each direction $e\in\mathbb{S}^{N-1}$, we know that there exists a pulsating travelling front $u=u_n(t,x)$ of \eqref{general-equation} with the speed $c=c_n$ such that
$$
0\leq\underline{u}_n(t,x)\leq u_n(t,x)
\leq\overline{u}_n(t,x)\leq1,~~
\forall(t,x)\in\mathbb{R}\times\mathbb{R}^N,~\forall n\in\mathbb{N}.
$$
Since $u_n(t,x)\to0$ (resp. $u_n(t,x)\to1$) as $t\to-\infty$ (resp. $t\to+\infty$) locally in $x$ and $u_n$ is increasing in $t$ (so that $0<u_n<1$), up to a translation in $t$, we can normalize the sequence of solutions such that $
\max_{x\in\mathbb{T}^N}u_n(0,x)=\frac{1}{2}$ for each $n\in\mathbb{N}$.  Passing to the limit $n\to\infty$ and from standard parabolic estimates, we obtain that the sequence $u_n$ converges (up to the extraction of a subsequence) in $C^{1,2}_{\rm loc}(\mathbb{R}\times\mathbb{R}^N)$ to an entire solution $u^*$ of \eqref{general-equation}. Moreover, the convergence also implies that $u^*$ satisfies
\begin{equation*}
 \begin{dcases}
 u^*\left(t+\frac{k\cdot e}{c},x\right)=u^*(t,x-k),~~\forall
   (t,x,k)\in\mathbb{R}\times\mathbb{R}^{N}\times\mathbb{Z}^{N}, \\
0\leq u^*(t,x)\leq\overline{u}^*(t,x):=
\min\left\{1,\Phi_{\lambda^*}(x)
e^{-\lambda^*\left(x\cdot e-c^*(e)t\right)}\right\},
~\forall(t,x)\in\mathbb{R}\times\mathbb{R}^N,\\
\partial_t u^*\geq0,~~\max_{x\in\mathbb{T}^N}u^*(0,x)=\frac{1}{2},
 \end{dcases}
\end{equation*}
where the second assertion holds due to Lemma \ref{supersolution} and Proposition \ref{minimal-speed}. Since $\overline{u}^*\not\equiv1$ is nondecreasing in $t$, it then follows that $u^*$ converges to $0$ as $t\to-\infty$ locally uniformly in $x\in\mathbb{R}^N$. Now, using the pulsating properties and the monotonicity of $u^*$ with respect to $t$, we obtain that $u^*$ converges in $C^2(\mathbb{T}^N)$ as $t\to+\infty$ to a periodic stationary state $0\leq p^*(x)\leq1$ of \eqref{general-equation}. Furthermore, the hypotheses \eqref{reaction-hypothesis} on $f$ imply that $p^*\equiv0$ or $p^*\equiv1$. But $p^*$ is nontrivial due to $\max_{x\in\mathbb{T}^N}u^*(0,x)=\frac{1}{2}$.
This shows that $u^*(t,x)\to1$ as $t\to+\infty$ locally uniformly in $x\in\mathbb{R}^N$.
Further, the positivity of $c^*(e)$ and the pulsating properties of $u^*$ with respect to $(t,x)$ imply that the limits of the front as in \eqref{PTF} hold.
The existence of a pulsating travelling front with speed $c^*(e)$ is proved.
\section*{Appendix}
Here we aim to find an orthogonal basis of $\mathbb{R}^N$ whose first vector is a rational point on $\mathbb{S}^{N-1}$ and then we give a new representation for any $N$-dimensional integer vector under this orthogonal basis.
\begin{proof}[Proof of Lemma \textup{\ref{change-basis}}]
Firstly, if $N=2$, let $\zeta_1=\zeta=(\frac{m_1}{n_1},\frac{m_2}{n_2})\in\mathbb{Q}^{2}\cap\mathbb{S}^{1}$ be given for $m_i\in\mathbb{Z}$ and $n_i\in\mathbb{Z}\backslash\{0\}$ and take some vector
$\zeta_2\in \zeta^{\perp}_{1}\backslash\{0_{\mathbb{R}^2}\}$. Note that
$\zeta^{\perp}_{1}=\mathbb{R}\zeta_2=\mathbb{R}(-\frac{m_2}{n_2},\frac{m_1}{n_1})$.
Therefore, for any given $k\in\mathbb{Z}^{2}$, there exist $p,\,q\in\mathbb{Z}$ such that
$k=\tau_{1}p\zeta_{1}+\tau_{2}q\zeta_2$ wherein $\tau_1=|n_1 n_2|^{-1}$ and $\tau_2=(|n_1 n_2|\cdot\lambda)^{-1}$ with $\lambda:=\|\zeta_2\|$. In particular, if $\zeta_1=e_1$ or $e_2$, it is sufficient to set $\zeta_2=e_2$ or $e_1$ and to take $\tau_1=\tau_2=1$. Consequently, we can rewrite
$$
\mathbb{Z}^{2}=\tau_1\mathbb{Z}\zeta_{1}\oplus\tau_2\mathbb{Z}\zeta_{2}
=\tau_1\mathbb{Z}\zeta_{1}\perp\tau_2\mathbb{Z}\zeta_{2}
$$
under the orthogonal basis $\{\zeta_1,\zeta_2\}$ of $\mathbb{R}^2$.

Next, if $N\geq3$, let $\zeta=(a_1,\ldots,a_N)\in\mathbb{Q}^{N}\cap\mathbb{S}^{N-1}$ be given and set
$$
I:=\{1\leq i\leq N\mid a_i\neq0\}~\mbox{ and }~
I^c:=\{1\leq i\leq N\mid a_i=0\}.
$$
Let us make the following decomposition
$$
\mathbb{Z}^N=\begin{pmatrix}{\mathbb{Z}^d}\\{0_{\mathbb{R}^{N-d}}}\end{pmatrix}
\bigoplus
\begin{pmatrix}{0_{\mathbb{R}^d}}\\{\mathbb{Z}^{N-d}}\end{pmatrix}
\mbox{ with }d=\card(I)
$$
and set $\beta_1:=(b_1,\ldots,b_d)$ wherein $b_j\in\{a_i\}_{i\in I}$ for all $j\in\{1,\ldots,d\}$ so that $\beta_1\in\mathbb{Q}^d\cap\mathbb{S}^{d-1}$. Now, we carry out a Gram-Schmidt process on the collection $\{\beta_1,e_1,\cdots,e_{d-1}\}$ of linearly independent vectors where $e_i$ denotes the standard coordinate vectors in $\mathbb{R}^d$. Then,
\begin{align*}
  &\beta_1=(b_1,\ldots,b_d),~~\|\beta_1\|=1,\\
  &\beta_2=\left(1-b_1^2,-b_1b_2,\ldots,-b_1b_d\right),~~
  \|\beta_2\|^2=1-b_1^2,\\
  &\beta_3=\frac{1}{\sum_{i=2}^{d}b_i^2}\left(0,\sum_{i=3}^{d}b_i^2,
  -b_2b_3,\ldots,-b_2b_d\right),~~\|\beta_3\|^2=
  \frac{\sum_{i=3}^{d}b_i^2}{\sum_{i=2}^{d}b_i^2},\\
  &\cdots\cdots\quad\cdots\cdots \\
  &\beta_{\ell}=\frac{1}{\sum_{i=\ell-1}^{d}b_i^2}
  \left(\underbrace{0,\ldots,0}_{\ell-2},
  \sum_{i=\ell}^{d}b_i^2,-b_{\ell-1}b_{\ell},\ldots,-b_{\ell-1}b_d\right),
  ~~\|\beta_{\ell}\|^2=\frac{\sum_{i=\ell}^{d}b_i^2}
  {\sum_{i=\ell-1}^{d}b_i^2},\\
  &\cdots\cdots\quad\cdots\cdots \\
  &\beta_{d}=\frac{1}{b_{d-1}^2+b_d^2}\left(0,\ldots,0,b_d^2,-b_{d-1}b_d\right),
  ~~\|\beta_d\|^2=\frac{b_d^2}{b_{d-1}^2+b_d^2}.
\end{align*}
For any given $k=(k_1,\ldots,k_d)\in\mathbb{Z}^d$, we are looking for a few suitable constants $\tau_i>0$ for all $i\in\{1,\ldots,d\}$ such that $k$ can be rewritten as $k=\sum^{d}_{i=1}\tau_i p_i\beta_i$ for $p_i\in\mathbb{Z}$. To that aim let us compute
\begin{align*}
  &k_1=\tau_1 p_1 b_1+\tau_2p_2(1-b_1^2),\\
  &k_2=\tau_1 p_1 b_2+\tau_2p_2(-b_1b_2)+
  \tau_3p_3\sum_{i=3}^{d}b_i^2\left(\sum_{i=2}^{d}b_i^2\right)^{-1},\\
  &\ldots\ldots\quad\ldots\ldots \\
  &k_\ell=\tau_1p_1b_\ell+\cdots+\tau_\ell p_\ell
  (-b_{\ell-1}b_\ell)\left(\sum_{i=\ell-1}^{d}b_i^2\right)^{-1}
  +\tau_{\ell+1}p_{\ell+1}
  \sum_{i=\ell+1}^{d}b_i^2\left(\sum_{i=\ell}^{d}b_i^2\right)^{-1},\\
  &\ldots\ldots\quad\ldots\ldots \\
  &k_d=\tau_1p_1b_d+\tau_2p_2(-b_1b_d)+\cdots+\tau_dp_d(-b_{d-1}b_d)
  \left(b_{d-1}^2+b_{d}^2\right)^{-1}.
\end{align*}
Since all of $\beta_i$ are orthogonal to each other, we have
\begin{align*}
  &\tau_1p_1=\sum_{i=1}^{d}b_ik_i,~~~
  \tau_2p_2\sum_{i=2}^{d}b_i^2=k_1\sum_{i=2}^{d}
  b_i^2-b_1\sum_{i=2}^{d}b_ik_i,
  ~~~~\ldots\ldots,\\
  &\tau_\ell p_\ell\sum_{i=\ell}^{d}b_i^2=k_{\ell-1}
  \sum_{i=\ell}^{d}b_i^2-b_{\ell-1}\sum_{i=\ell}^{d}
  b_ik_i,~~~\ldots\ldots,~~~
  \tau_dp_db_d=b_dk_{d-1}-b_{d-1}k_d.
\end{align*}
Note that $\beta_1=(b_1,\ldots,b_d)\in\mathbb{Q}^{d}\cap\mathbb{S}^{d-1}$. Let us set $b_i:=\frac{m_i}{n_i}$ with $m_i,n_i\in\mathbb{Z}\backslash\{0\}$ for all $i\in\{1,\ldots,d\}$. A straightforward computation yields that we can choose
\begin{align*}
  &\tau_1=\left(\prod_{i=1}^{d}|n_i|\right)^{-1},~~
\tau_2=\left[(n_1^2-m_1^2)\prod_{i=2}^{d}|n_i|\right]^{-1},~~\ldots,\\
&\tau_\ell=\left[|n_{\ell-1}|\sum_{i=\ell}^{d}\left(\frac{m_i}{n_i}\right)^2
\prod_{i=\ell}^{d}n^2_i\right]^{-1},~~\ldots,~~\tau_d=|m_dn_{d-1}|^{-1}
\end{align*}
such that $k=\sum^{d}_{i=1}\tau_i p_i\beta_i$ and $p_i\in\mathbb{Z}$ where the choice for each $\tau_i$ is not necessarily unique. Now, for all $1\leq i\leq d$, we set $\eta_i:=\lambda_i\beta_i$ with some $\lambda_i\in\mathbb{R}\backslash\{0\}$. As argued above, we know that $\eta_1^{\perp}=\textup{span}\{\eta_i\}_{i=2}^{d}$ and for any given $k\in\mathbb{Z}^d$, there exist $d$ constants $\tau^{\prime}_i>0$ such that
$$
k=\sum^{d}_{i=1}\tau^{\prime}_i p_i\eta_i, ~~p_i\in\mathbb{Z}.
$$
In particular, if $\lambda_i=1/\|\beta_i\|$ for each $i\in\{1,\ldots,d\}$, then $\{\eta_i\}_{i=1}^d$ forms an orthonormal basis of $\mathbb{R}^d$. This means that any integer vector in $\mathbb{R}^d$ under the canonical basis $\{e_i\}_{i=1}^d$ can be expressed uniquely as a linear combination of the orthogonal basis $\{\eta_i\}_{i=1}^d$. In other words, we can always find an orthogonal basis $\{\eta_i\}_{i=1}^d$ of $\mathbb{R}^d$ and determine $d$ constants $\tau^{\prime}_i>0$ such that $\eta_1=\beta_1\in\mathbb{Q}^{d}\cap\mathbb{S}^{d-1}$ and
\begin{align}\label{proof-change-basis}
  \mathbb{Z}^d=\bigoplus_{i=1}^d\tau^{\prime}_i\mathbb{Z}\eta_i,
~~\eta_{i}\perp\eta_{j}\,(i\neq j), ~~i,j=1,\ldots,d.
\end{align}

Finally, $\{\eta_i\}_{i=1}^d$ can be easily expanded to a collection $\{\zeta_\ell\}_{\ell\in I}$ of vectors orthogonal to each other in $\mathbb{R}^N$. Therefore, the set $\{\zeta_\ell\}_{\ell\in I}\cup\{e_m\}_{m\in I^c}$ forms an orthogonal basis of $\mathbb{R}^N$ and $\zeta=\zeta_{i_1}$ where $i_1\in I$ denotes the first index such that $a_i\neq0$ for all $i\in\{1,\ldots,N\}$. Taking $\tau_\ell=\tau^{\prime}_\ell$ for all $\ell\in I$ and taking $\tau_m=1$ for all $m\in I^c$, possibly along some rearrangement of subscripts, we can obtain from \eqref{proof-change-basis} that $\zeta_1=\zeta\in\mathbb{Q}^N\cap\mathbb{S}^{N-1}$ and
$$
\mathbb{Z}^N=\begin{pmatrix}{\mathbb{Z}^d}\\{0_{\mathbb{R}^{N-d}}}\end{pmatrix}
\bigoplus
\begin{pmatrix}{0_{\mathbb{R}^d}}\\{\mathbb{Z}^{N-d}}\end{pmatrix}=
\bigoplus_{i=1}^N\tau_i\mathbb{Z}\zeta_i,
  ~~\zeta_{i}\perp\zeta_{j}~ (i\neq j),~~i,j=1,\ldots,N.
$$
In particular, if $d=1$, namely $\zeta=e_{i_0}$ for some $i_0$, it is sufficient to choose $\zeta_1=e_{i_0}$ and  $\zeta_i\in\{e_j\}_{j=1}^N\backslash\{e_{i_0}\}$ for all $i\in\{2,\ldots,N\}$, and to take $\tau_\ell=1$ for all $\ell\in\{1,\ldots,N\}$.
This ends the proof of Lemma \ref{change-basis}.
\end{proof}

\end{document}